\documentclass[12pt]{amsart}
\usepackage{amssymb}
\usepackage{latexsym}
\usepackage{amsmath}
\usepackage{bm}

\def\wt{\widetilde}
\def\wh{\widehat}
\def\ov{\overline}
\def \im{{\rm Im\,}}
 \def\up{\upharpoonright}

\def\cA{\mathcal A} \def\cH{\mathcal H}  \def\cB{\mathcal B} \def\cC{\mathcal C}
\def\cD{\mathcal D} \def\cF {\mathcal F}   

\def\cK{\mathcal K} \def\cL{\mathcal L}
 \def\cN{\mathcal N} 
  \def\cS{\mathcal S}
 
  \def\cI{\mathcal I}

\def\B{\mbox{\boldmath$B$}}

\def \gH{\mathfrak H}   \def \gN{\mathfrak N}

\def \bC{\mathbb C} \def \bN{\mathbb N} \def \bZ{\mathbb Z}
\def\bR{\mathbb R}
 
\def \l{\lambda}
\def \a{\alpha} \def \b{\beta}   \def \L{\Lambda}  \def \s{\sigma} \def \t{\theta}
 \def\g {\gamma}
\def\d {\delta}

\def \f{\varphi} \def\D {\Delta} 
 \def \G{\Gamma} 
\def \C{\widetilde {\mathcal C}}
 \def \CAt {C(\wt A_\tau)}

\def \cd {\cdot}

\def\Rm {R_{\rm mer}} \def\RM {\wt R_{\rm mer}}
\def\nt {\wt\gN_t}  

\def\EP {{\rm ENP}(\cH)}

\def\lI {\cL_\Delta^2(\cI)} \def\LI {L_\Delta^2(\cI)}

\def\AC {AC(\cI)}

\def\lS {\cL^2(\xi; \bC^r)} \def\LS {L^2(\xi ; \bC^r)}

   \def\Tmi{T_{\min}}

  \def\Sma{S_{\max}} \def\Smi{S_{\min}}
\def\Dma{\cD_{\rm max}} \def\Dmi{\cD_{\rm min}}

\def \ex { {\rm ext} (A)} \def \cex {\overline {\rm ext }(A)}

\def \SA {S_{\wt A}}

\def \dom {{\rm dom}\,}  \def \ran {{\rm ran}\,}  \def \ker{{\rm
ker\,}}
 \def \mul {{\rm mul}\,}

\def  \RH {\wt R (\cH)}  \def  \RCH {\wt R_c (\cH)}

\def \CR {\bC\setminus\bR}

\def\bt{\{\cH,\G_0,\G_1\}}

\newcommand{\Romannumeral}[1]{\uppercase\expandafter{\romannumeral #1\relax}}

\newtheorem{theorem}{Theorem}[section]
\newtheorem{proposition}[theorem]{Proposition}
\newtheorem{corollary}[theorem]{Corollary}
\newtheorem{lemma}[theorem]{Lemma}
\newtheorem{assertion}[theorem]{Assertion}
\theoremstyle{definition}

\theoremstyle{definition}
\newtheorem {definition} [theorem]{Definition}
\theoremstyle{remark}
\newtheorem{remark}[theorem]{Remark}
\numberwithin{equation}{section}
\begin{document}
\title[On eigenfunction expansions ]
{On eigenfunction expansions of differential equations with degenerating  weight}
\author{Vadim  Mogilevskii}
\address{Department of Mathematical Analysis and Informatics, Poltava National V.G. Korolenko Pedagogical University,  Ostrogradski Str. 2, 36000 Poltava, Ukraine }
\email{vadim.mogilevskii@gmail.com}

\begin{abstract}
Let  $A$ be a symmetric  operator. By using the method of boundary triplets we parameterize in  terms of a Nevanlinna parameter $\tau$ all exit space extensions $\wt A=\wt A^*$ of $A$ with the discrete spectrum $\s(\wt A)$ and  characterize the Shtraus family of $\wt A$  in terms of abstract boundary conditions.  Next we apply these results to the eigenvalue problem for the $2r$-th order differential equation $l[y]= \l \D(x)y$ on an interval $[a,b), \; -\infty <a<b\leq \infty,$ subject to $\l$-depending separated boundary conditions with entire operator-functions $C_0(\l) $ and $C_1(\l)$, which form a Nevanlinna pair $(C_0,C_1)$. The weight $\D(x)$ is nonnegative and may vanish on some intervals $(\a,\b)\subset \cI$.  We show that in the case when the minimal  operator  of the equation has the discrete spectrum (in particular, in the case of the quasiregular equation) the set of eigenvalues of the eigenvalue  problem is an infinite  subset of $\bR$ without finite limit points and  each function $y\in\LI$ admits the eigenfunction expansion $y(x)=\sum_{k=1}^\infty y_k (x)$ converging in $\LI$. Moreover,  we give an explicit method for calculation of eigenfunctions  $y_k$  in this expansion  and specify boundary conditions on $y$ implying the uniform convergence of the eigenfunction expansion of y. These results develop the known  ones obtained for the case of the positive weight $\D$ and the more restrictive class of $\l$-depending boundary conditions.
\end{abstract}
\maketitle
\section{Introduction}
We study the eigenvalue problem  for the differential equation of an even order $2r$
\begin{gather}\label{1.1}
l[y]= \sum_{k=0}^r  (-1)^k \left( p_{r-k}(x)y^{(k)}\right)^{(k)}=\l \D(x)y,\quad x\in \cI=[a,b\rangle , \quad -\infty <a<b\leq \infty
\end{gather}
subject to separated boundary conditions
\begin{gather}\label{1.2}
(\cos B) y^{(1)}(a)+(\sin B) y^{(2)}(a)=0, \qquad C_0(\l)\G_{0b}y+C_1 (\l) \G_{1b}y =0
\end{gather}
depending on the parameter $\l\in\bC$.  It is assumed that the coefficients $p_j$ and the weight $\D$ in \eqref{1.1} are real-valued functions on an interval $\cI=[a,b\rangle $ such that $\D(x)\geq 0$ a.e.on $\cI$ and $p_0^{-1}, \; p_1, \dots, p_r, \D$ are  integrable on each compact interval $[a,b']\subset \cI$ (the latter means that the endpoint $a$ is regular). Below we denote by $\B (\bC^m)$ the set of all linear operators in $\bC^m$ or equivalently the set of all $m\times m$-matrices. In \eqref{1.2} $B=B^*\in\B (\bC^r)$,  $y^{(1)}$ and $ y^{(2)}$ are vectors of quasi-derivatives  of a function $y$, $\G_{jb}y\in\bC^m$  are singular boundary values of $y^{(j)}$ at the endpoint $b$ and $C_0, C_1$ are entire $\B (\bC^{m})$-valued functions which form a Nevanlinna pair $(C_0,C_1)$. The particular case of \eqref{1.2} are  the following selfadjoint separated boundary conditions with $B_1=B_1^*\in\B (\bC^{d-r})$:
\begin{gather}\label{1.3}
(\cos B) y^{(1)}(a)+(\sin B) y^{(2)}(a)=0, \qquad (\cos B_1)\G_{0b}y +(\sin B_1) \G_{1b}y =0
\end{gather}
Note that eigenfunction expansions generated by problems of the type \eqref{1.1}, \eqref{1.2} have been studied by many authors (see, e.g., \cite{Bin02,Dij80,Ful77,Hin79,Wal73} and their references).

Denote by $\LI$ the Hilbert space of all  functions $f$ on $\cI$ such that $\int\limits_{\cI}\D (x)|f(x)|^2 \, dx <\infty$ and by $(\cd,\cd)_\D$ the inner product in $\LI$. Recall that equation \eqref{1.1} is called quasiregular if each solution $y$ of \eqref{1.1} belongs to $\LI$ and regular if $\cI=[a,b]$ is a compact interval (clearly, each regular equation is quasiregular). For regular equation one can put in \eqref{1.2} and \eqref{1.3} $\G_{0b}y =y^{(1)}(b)$ and $\G_{1b}y =y^{(2)}(b)$.
\begin{definition}\label{def1.0}
The weight  $\D$ in \eqref{1.1} is called  positive, if $\D(x)>0$ a.e. on $\cI$ and nontrivial,  if the set $\{x\in\cI:\D(x)>0\}$ has the positive Borel measure.
\end{definition}
Clearly, non triviality  is the weakest restriction on $\D$, which saves the interest to studying of \eqref{1.1}.

According to \cite{Mog20}  boundary problem \eqref{1.1}, \eqref{1.3} for equation \eqref{1.1} with the nontrivial weight  $\D$ generates an  operator $T=T^*$ in $\LI$ (for the case of the positive Weight see e.g. \cite{Wei}). Equation \eqref{1.1} is said to have the discrete spectrum if for some (and hence any) boundary problem \eqref{1.1}, \eqref{1.3} the   operator  $T$ has the discrete spectrum.

An eigenvalue of the problem \eqref{1.1}, \eqref{1.2} is defined as $\l\in\bC$  for which this problem has a solution $y\in\LI,\, y\neq 0$; this solution is called an eigenfunction.  The following eigenfunction expansion theorem is the classical result for the selfadjoint  eigenvalue problem \eqref{1.1}, \eqref{1.3} (see e.g \cite{Col,DunSch}).
\begin{theorem}\label{th1.1}
Assume that equation \eqref{1.1} with the positive weight $\D$ has the discrete spectrum. Then:

{\rm (i)} The set  of all eigenvalues of the problem \eqref{1.1}, \eqref{1.3}  is an infinite  countable subset $\{t_k\}$ of $\bR$ without finite limit points. Moreover, eigenfunctions for different $t_k$ are mutually orthogonal in $\LI$.

{\rm (ii)} Each absolutely continuous  function $y\in\LI$ with absolutely continuous quasiderivatives $y^{[k]}$ satisfying   $\D^{-1}l[y]\in\LI$ and the boundary conditions \eqref{1.3} admits the eigenfunction expansion
\begin{gather}\label{1.4}
y(x)=\sum_{k=1}^{\infty}y_k(x),
\end{gather}
which converges absolutely and uniformly on each compact interval $[a,b']\subset \cI$. In \eqref{1.4} $y_k(x)=(y,v_k)_\D v_k(x)$, where $\{v_k\}$ is an orthonormal  system of eigenfunctions.
\end{theorem}
F. Atkinson in \cite[Theorem 8.9.1]{Atk} extended Theorem \ref{th1.1} to regular Sturm-Liouville equations
\begin{gather}\label{1.6}
l[y]=-(p(x)y')' + q(x)y=\l \D(x)y,\quad x\in \cI=[a,b\rangle
\end{gather}
such that $0\leq p(x)\leq \infty$ and the  following conditions are satisfied: (i) $\D(x)\geq 0$ a.e. on $\cI$; (ii) there is no  interval $(a,b')\subset\cI$ such that $\D(x)=0$ a.e. on  $(a,b')$ and there is no interval $(a',b)\subset \cI$ such that $\D(x)=0$ a.e. on $(a',b)$; (iii) if $\D(x)=0$ a.e. on an interval $(a',b')\subset \cI$, then $q(x)=0 $ a.e. on $(a',b')$.

Eigenfunction expansions for the regular Sturm-Liouville  equations \eqref{1.6} with the positive weight $\D$ subject to the $\l$-depending boundary conditions of the special form
\begin{gather}\label{1.7}
\cos B\cd y(a)+\sin B\cd (py')(a)=0, \quad  (M_0-\l N_0)y(b) + (M_1-\l N_1)\cd (py')(b) =0
\end{gather}
were studied in the papers \cite{Wal73,Ful77,Hin79}. It is assumed there that $p>0$ and the coefficients $M_j$ and $N_j$ in \eqref{1.7} are such that $(M_0-\l N_0)(M_1-\l N_1)^{-1}$ is a Nevanlinna function. In these papers problem \eqref{1.6}, \eqref{1.7} is associated with a certain self-adjoint operator $\wt A$ in a Hilbert space $H\supset \LI$. It follows from the results of \cite{Ful77,Hin79} that eigenvalues of the problem  \eqref{1.6}, \eqref{1.7} form a strictly increasing unbounded sequence $\{\l_k\}, \; \l_k\in\bR,$  and each function $y\in\AC$  such that  $py'\in\AC$  and $\D^{-1}l[y]\in\LI$ admits the eigenfunction expansion  \eqref{1.4} converging absolutely and uniformly on $\cI$. The results of \cite{Ful77,Hin79} were extended to boundary problems for  regular equations  \eqref{1.1} of an order $2r$ with the positive weight $\D$ subject to linear  boundary conditions of the type \eqref{1.7} in \cite{Bin02} and boundary conditions \eqref{1.2} with polynomial matrices $C_j(\l)$ in \cite{Dij80} . Note that eigenfunctions of the problem \eqref{1.6}, \eqref{1.7}  are not mutually orthogonal and the method of  \cite{Bin02,Dij80,Ful77,Hin79} does not give explicit formulas for calculation of $y_k$ in \eqref{1.4}.

In the present paper the above results are developed  in the following two directions:

(i) Instead of positivity we assume that the weight $\D\geq 0$ in \eqref{1.1} is nontrivial, so that  the equality $\D=0$ may holds on some intervals $(\a,\b)\subset \cI$.

(ii) It is assumed that $C_j(\l)$ in the boundary conditions \eqref{1.2} are entire operator functions which form a Nevanlinna pair $(C_0,C_1)$.  This assumption generalizes the known results  to a wider class of $\l$-depending boundary conditions.

In the case of the Sturm-Liouville equation \eqref{1.6} our main results can be formulated in the form of the following two theorems.
\begin{theorem}\label{th1.2}
Assume that  equation \eqref{1.6} with the nontrivial weight $\D$  is  quasiregular. Let $\Dma$ be the set of all functions $y\in\AC\cap\LI$ such that $y^{[1]}:=py'\in\AC$ and $-(y^{[1]})'+q y=\D(x) f_y$ (a.e.on $\cI$) with some $f_y\in\LI$. For a given $B\in\bR$ denote by $\f_B(\cd,\l)$ and $\psi_B(\cd,\l)$ the solutions of \eqref{1.6} with $\f_B(a,\l)=\sin B, \; \f_B^{[1]}(a,\l)=-\cos B$ and $\psi_B(a,\l)=\cos B, \; \psi_B^{[1]}(a,\l)=\sin B$. Denote also by $\G_{0b}y$ and $\G_{1b}y$ singular boundary values of a function $y\in\Dma$ at the endpoint $b$ given by
\begin{gather*}
\G_{0b}y=\lim_{x\to b} (\psi_B^{[1]}(x,0)y(x) - \psi_B(x,0)y^{[1]}(x)),\;\;
\G_{1b}y=\lim_{x\to b} (-\f_B^{[1]}(x,0)y(x) + \f_B(x,0)y^{[1]}(x)).
\end{gather*}
Let $C=(C_0,C_1)$ be an entire Nevanlinna pair, i.e., a pair of entire functions $C_0$ and $C_1$ without common zeros and such that
\begin{gather*}
 \quad \im \l \cd \im (C_1(\l)\ov{C_0(\l)})\geq 0, \quad C_1(\ov \l)\ov{C_0(\l)}-\ov {C_1( \l)} C_0 (\ov\l)=0,\quad \l\in\CR.
\end{gather*}
Consider the eigenvalue problem for equation \eqref{1.6} subject to the boundary conditions
\begin{gather}\label{1.11}
\cos B\cd y(a)+\sin B\cd y^{[1]}(a)=0, \qquad C_0(\l)\G_{0b}y +C_1(\l)\G_{1b}y=0.
\end{gather}
and let $ EV$ be the set of all eigenfunctions of this problem. Then:

{\rm (i)} $ EV=\{t_k\}$ is an infinite countable subset of $\bR$ without finite limit points and the linear-fractional transform
\begin{gather} \label{1.12}
m(\l)= \frac
{w_2(\l)C_0(\l)+w_4(\l)C_1(\l)} {w_1(\l)C_0(\l)+w_3(\l)C_1(\l)}, \quad \l\in\CR
\end{gather}
with the coefficients
\begin{gather*}
w_1(\l)=1+\l\smallint_\cI \psi_B(x,0)\D(x) \f_B(x,\l)\,dx, \qquad w_2(\l)=\l\smallint_\cI \psi_B(x,0)\D(x) \psi_B(x,\l)\,dx\\
w_3(\l)=-\l\smallint_\cI \f_B(x,0)\D(x) \f_B(x,\l)\,dx, \qquad w_4(\l)=1-\l\smallint_\cI \f_B(x,0)\D(x) \psi_B(x,\l)\,dx.
\end{gather*}
defines a meromorphic Nevanlinna function $m$ such that $ EV$ coincides with the set of all poles of $m$.

{\rm (ii)} Assume that  $y\in\LI$, $\wh y_k=\int\limits_\cI \f_B(x,t_k)\D(x)y(x)\, dx$ is the Fourier coefficient of $y$ and $\wh\xi_k$ is the residue of $m$ at the pole $t_k$. Then the equality
\begin{gather}\label{1.12.1}
y_k(x)=-\wh \xi_k\,\wh y_k\,\f_B(x,t_k)
\end{gather}
defines eigenfunctions $y_k$ of the problem \eqref{1.6}, \eqref{1.11} such that  eigenfunction expansion \eqref{1.4} of $y$ converging in $\LI$ is valid.
\end{theorem}
If $C=(C_0,C_1)$ is an entire Nevanlinna pair and $C_1(\l)\neq 0,\;\l\in\CR$, then $\tau:=-C_0 C_1^{-1} $ is a  Nevanlinna function. Therefore there exist the limits $\qquad\cB_\infty=\lim\limits_{y\to \infty}\tfrac 1 {iy} \tau (iy)\in \bR$,  $ \;\wh D_\infty=\lim\limits_{y\to \infty} y\im \tau(i y)\leq \infty$ and one of the following alternative cases holds:

Case 1.  $\;\cB_\infty\neq 0$;\;\; Case 2. $\; \cB_\infty = 0$  and $\wh D_\infty <\infty$; \;\; Case 3. $\;\cB_\infty = 0$ and $\wh D_\infty =\infty$.
\vskip 1mm
\noindent Moreover, in Case 2 there exists the limit $D_\infty= \lim\limits_{y\to\infty} \tau(iy) $.
\begin{theorem}\label{th1.3}
Let under the assumptions of Theorem \ref{th1.2} $C_1(\l)\neq 0, \; \l\in\CR,$ and let  a function $y\in\Dma$ satisfies one of the following boundary conditions \eqref{1.13} -- \eqref{1.15} depending on Cases 1 -- 3:
\begin{gather}
\text{in Case 1}:\quad \cos B\cd y(a)+\sin B\cd y^{[1]}(a)=0, \quad  \G_{0b}y =0;\label{1.13}\\
\text{in Case 2}:\quad \cos B\cd y(a)+\sin B\cd y^{[1]}(a)=0, \quad  \G_{1b}y=D_\infty\G_{0b}y ;\label{1.14}\\
\text{in Case 2}:\quad\cos B\cd y(a)+\sin B\cd y^{[1]}(a)=0, \quad \G_{0b}y =\G_{1b}y=0.\label{1.15}
\end{gather}
Then the eigenfunction expansion \eqref{1.4} of $y$ converges absolutely  and uniformly on each compact interval $[a,b']\subset \cI$.
\end{theorem}
\begin{remark}\label{rem1.4}
If in addition to the assumptions of Theorems \ref{th1.2} and \ref{th1.3} the equation \eqref{1.6} is regular, then statements of this theorems are valid  with $\G_{0b}y=y(b)$, $\G_{1b}y =y^{[1]}(b)$ and  the equality \eqref{1.12} should be replaced with
\begin{gather}\label{1.16}
m(\l)= \frac
{\psi_B(b,\l)\,C_0(\l)+\psi_B^{[1]}(b,\l)\,C_1(\l)} {\f_B(b,\l)\,C_0(\l)+\f_B^{[1]}(b,\l\,)C_1(\l)}, \quad \l\in\CR.
\end{gather}
\end{remark}
Similar theorems are proved in the paper for equation \eqref{1.1} of an even order $2r$ with the nontrivial  weight $\D$ and the discrete spectrum.  We study eigenvalue  problem for this equation subject to the boundary conditions \eqref{1.2} with an entire Nevanlinna pair $C=(C_0, C_1)$, i.e, a pair of entire operator functions $C_0$ and $C_1$ such that $\ran C(\l)=\bC^m$, $i\im \l\cd C(\l)J C^*(\l)\leq 0$ and $C(\l)J C^*(\ov\l)=0, \;\l\in\bC,$ with $J=\begin{pmatrix}  0 &-I_m\cr I_m & 0 \end{pmatrix}$. The  operator functions $C_0$ and $C_1$ take on values  in $\B (\bC^m)$, where $m=d-r$ and $d$ is the deficiency index of the equation \eqref{1.1}. We show that the set of eigenvalues of the problem   \eqref{1.1}, \eqref{1.2}  is an infinite countable subset of $\bR$ without finite limit points and  each function $y\in\LI$ admits the eigenfunction expansion \eqref{1.4} converging in $\LI$ (the eigenfunctions $y_k$ in \eqref{1.4} are not generally speaking mutually  orthogonal). Moreover, similarly to Theorems \ref{th1.2} and \ref{th1.3} we give an explicit method for calculation of $y_k$  in terms of the original entire Nevanlinna pair $C=(C_0,C_1)$ and specify boundary conditions on $y$ implying the uniform convergence of \eqref{1.4} (see Theorems \ref{th5.12}, \ref{th5.17} and \ref{th5.19}). In the case of the selfadjoint eigenvalue problem \eqref{1.1}, \eqref{1.3} these results extend Theorem \ref{th1.1} to equations \eqref{1.1} with the nontrivial weight $\D$.

Our approach  is based on the method of boundary triplets in the extension theory of a symmetric  operator $A$ in the Hilbert space $\gH$ (see \cite{BHS,DM17,GorGor} and references therein). By using this method  we parameterize in  terms of a Nevanlinna parameter $\tau$ all exit space extensions $\wt A=\wt A^*$ of $A$ with the discrete spectrum $\s(\wt A)$. Moreover, we characterize the Shtraus family of the extension $\wt A$ \cite{BHS,Sht70} in terms of abstract boundary conditions and describe extensions $\wt A$ with certain special properties of their Shtraus family. These results enables  us to define an abstract eigenvalue  problem for $A^*$, which generates the abstract eigenvector expansion  in $\gH$. In the case when $A$ is the minimal operator of the equation \eqref{1.1} the abstract eigenvalue problem for $A^*$  takes the form \eqref{1.1}, \eqref{1.2} and the abstract eigenvector expansion turns into \eqref{1.4}.

In conclusion note that this paper is a continuation of \cite{Mog20}, where ''continuous'' eigenfunction expansions in the form of the generalized Fourier transform were considered.

\section{Preliminaries}
\subsection{Some notations and definitions}
The following notations will be used throughout the paper: $\gH$, $\cH$ denote separable  Hilbert spaces; $\B (\cH_1,\cH_2)$  is the set of all bounded linear operators defined on $\cH_1$ with values in  $\cH_2$; $\B(\cH)=\B(\cH,\cH)$; $A\up \mathcal L$ is a restriction of the operator $A\in \B(\cH_1,\cH_2)$ to the linear manifold $\mathcal L\subset\cH_1$; $P_\cL$ is the orthoprojection in $\gH$ onto the subspace $\cL\subset \gH$;  $\bC_+\,(\bC_-)$ is the open  upper (lower) half-plane  of the complex plane.

A non-decreasing strongly left continuous operator function $\xi: \bR\to \B(\cH)$ with $\xi(0)=0$ is called a distribution function or shortly a distribution.

We denote by $\cF$ the class of all sets $F\subset \bR$ satisfying at least one (and hence all) of the following equivalent  conditions: (i) the set $F\cap [a,b]$ is finite or empty for any compact interval $[a,b]\subset \bR$; (ii) $F$ is closed and all points of $F$ are isolated; (iii) the set $F$ has no finite limit points. Clearly, each set $F\in\cF$ admits the representation  in the form of a finite or infinite increasing  sequence $F=\{t_k\}_{\nu_-}^ {\nu_+}$, where $\nu_-\in\bZ\cup \{-\infty\}, \; \nu_+\in\bZ\cup \{\infty\}$, $-\infty\leq \nu_-\leq \nu_+\leq\infty$ and $k$ is an integer between $\nu_-$ and $\nu_+$. We assume that $\emptyset\in\cF$.

Recall that a linear manifold $T$ in the Hilbert space $\cH_0\oplus\cH_1$ ($\cH\oplus\cH$) is called a  linear relation from $\cH_0$ to $\cH_1$ (resp. in $\cH$). The set of all closed linear relations from $\cH_0$ to $\cH_1$ (in $\cH$) will be denoted by $\C (\cH_0,\cH_1)$ (resp. $\C(\cH)$). Clearly for each linear operator $T:\dom T\to\cH_1, \;\dom T\subset \cH_0,$ its graph ${\rm gr} T =\{\{f,Tf\}:f\in \dom T\} $  is a linear relation from $\cH_0$ to $\cH_1$. This fact enables one to consider an operator $T$ as a linear relation. In the following we denote by $\mathcal C (\cH_0,\cH_1)$ the set of all closed linear operators $T:\dom T\to \cH_1$ with $\dom T\subset \cH_0$. Moreover, we let $\mathcal C (\cH)=\cC (\cH,\cH)$.

For a linear relation $T$ from $\cH_0$ to $\cH_1$  we denote by $\dom T,\; \ker T, \; \ran T$ and  $\mul T:=\{h_1\in\cH_1: \{0,h_1\}\in T\}$
the domain, kernel, range and multivalued part  of $T$ respectively; moreover, we let $\wh\mul T=\{0\}\oplus\mul T$. Clearly $T$ is an operator if and only if $\mul T=\{0\}$. Denote also by $T^{-1}$ and $T^*$ the inverse and adjoint linear relations of $T$ respectively.

A linear relation $T$ in $\gH$ is called symmetric (self-adjoint) if $T\subset T^*$ (resp. $T=T^*$). As is known, for a symmetric relation $T\in\C (\gH)$  the decompositions
\begin{gather}\label{2.7}
\gH=\gH_0\oplus \mul T, \qquad  T= {\rm gr}\, T_{\rm op}\oplus \wh\mul T
\end{gather}
hold with $\gH_0=\gH\ominus \mul T$ and a symmetric operator $T_{\rm op}\in\cC (\gH_0)$ (the operator part of $T$). Clearly, $T^*=T$ if and only if $T_{\rm op}^*=T_{\rm op}$.

For a linear relation $T\in\C (\gH)$ and $\l\in\bC$ we let
\begin{gather*}
\gN_\l(T)=\ker (T-\l)=\{f\in\gH:\{f,\l f\}\in T\}, \quad \wh\gN_\l(T)=\{\{f,\l f\}:f\in\gN_\l(T)\}.
\end{gather*}
Moreover, we use the following notations: $\rho(T)=\{\l\in\bC:(T-\l)^{-1}\in\B(\gH)\}$ is the  resolvent set, $\s(T)=\bC\setminus \rho (T)$ is the spectrum,  $\s_p(T)=\{\l\in\bC:\gN_\l(T)\neq 0\}$ is the point spectrum  and $\wh\rho(T)=\{\l\in\bC: \gN_\l(T)=\{0\} \;\; {\rm and}\;\; \ran (T-\l)\;\;\text{\rm  is closed}\}$
is the set of regular type points of $T$.

A subspace $\gH_1\subset \gH$ reduces a relation $T\in\C (\gH)$ if $T=T_1\oplus T_2$,  where $T_1\in \C (\gH_1)$ and $T_2\in \C (\gH\ominus\gH_1)$.

For an operator $T=T^*\in \B(\gH)$ we write $T\geq 0$ if $(Tf,f)\geq 0,\; f\in \gH,$ and $T>0 $ if $T-\a I\geq 0$ with some $\a>0$.
\subsection{Nevanlinna functions}
Recall that  a holomorphic operator function $\tau:\CR\to \B(\cH)$ is called a Nevanlinna function if $\im\l\cd\im \tau (\l)\geq 0$ and $\tau^*(\l)=\tau (\ov\l), \; \l\in\CR$. The class of all  Nevanlinna $\B(\cH)$-valued functions will be denoted by $R[\cH]$.

As is known (see e.g. \cite{BHS,DM17}), an operator function $\tau:\CR\to \B (\cH)$ belongs to the class $R[\cH]$ if and only if it admits the integral representation
\begin{gather}\label{2.9}
\tau(\l)=\cA + \l \cB + \int_\bR \left (\frac 1 {t-\l}- \frac t {t^2+1}\right) \, d\xi(t), \quad \l\in\CR,
\end{gather}
with self-adjoint operators $\cA,\;\cB\in \B(\cH), \; \cB\geq 0,$ and a distribution  $\xi:\bR\to \B(\cH)$ such that $\int\limits_\bR \left(t^2+1 \right )^{-1}(d\xi(t)h,h) <\infty, \; h\in\cH $. The distribution  $\xi$ in \eqref{2.9} is called a spectral function of $\tau$ (it defines uniquely by $\tau$).

The operator-function $\tau\in R[\cH]$ belongs to the class: (i) $R_c[\cH]$, if $\ran \im \tau(\l)$ is closed for all $\l\in\CR$; (ii)  $R_u[\cH]$ if $\im\l\cd\im \tau (\l) > 0, \; \l\in\CR$. Clearly, in the case $\dim \cH<\infty $ one has   $R_c [\cH]=R [\cH]$.

The following proposition is well known (see e.g. \cite{Mal92}).
\begin{proposition}\label{pr2.4}
Let  $\tau\in R[\cH]$. Then:

{\rm (i)} The equalities
\begin{gather}
\cB_{\tau \infty}=s\text{-}\lim_{y\to\infty}\tfrac 1 {iy}\tau(iy), \nonumber\\
\dom D_{\tau\infty} =\{h\in\cH: \lim_{y\to\infty}  y\im (\tau(iy)h,h)< \infty\}, \;\; D_{\tau\infty} h=\lim_{y\to\infty}  \tau(iy)h, \quad h\in \dom D_{\tau\infty}\nonumber
\end{gather}
correctly define  the operator $\cB_{\tau\infty}\in \B(\cH),\; \cB_{\tau\infty}\geq 0$, the (not necessarily closed) linear manifold $\dom D_{\tau\infty} \subset \cH$ and the  linear operator $D_{\tau\infty}:\dom D_{\tau\infty}\to \cH$. Moreover, $\dom D_{\tau\infty} \subset \ker\cB_{\tau\infty}$.

{\rm (ii)} For each $t\in\bR$ the equalities
\begin{gather}
\cB_{\tau}(t)=s\text{-}\lim_{y\to 0}(-iy\,\tau(t+iy)),  \qquad
\dom D_{\tau}(t) =\{h\in\cH: \lim_{y\to 0}  \tfrac 1 y \im (\tau(t+iy)h,h)< \infty\}\nonumber\\
 D_{\tau}(t) h=\lim_{y\to 0}  \tau(t+iy)h, \quad h\in \dom D_{\tau}(t)\nonumber
\end{gather}
correctly define  the operator $\cB_{\tau}(t)\in \B(\cH),\; \cB_{\tau}(t)\geq 0$, the (not necessarily closed) linear manifold $\dom D_{\tau}(t) \subset \cH$ and the  linear operator $D_{\tau}(t):\dom D_{\tau}(t)\to \cH$.
\end{proposition}
\begin{definition}\label{def2.4.1}$\,$\cite{DM17,BHS}
A function $\tau:\CR\to \C (\cH)$ is referred to the class $\RH$ of Nevanlinna relation valued functions if: (\romannumeral 1) $\im (f',f)\geq 0, \quad \{f,f'\}\in\tau (\l),\quad \l\in\bC_+$; (\romannumeral 2)  $(\tau(\l)+i)^{-1}\in \B (\cH)$  for any $\l\in\bC_+$ and $(\tau (\l)+i)^{-1}$ is a holomorphic operator-function in $\bC_+$;
(\romannumeral 3) $\tau^*(\l)=\tau (\ov\l),\;\l\in\CR$.
\end{definition}

It is clear that $R [\cH]\subset \wt R(\cH)$.

According to \cite{KreLan71} for each function $\tau\in \RH$ the multivalued part $\cK:=\mul \tau(\l)$ of $\tau (\l)$ does not depend on $\l\in\CR$ and for each such $\l$ the decompositions
\begin{equation}\label{2.11}
\cH=\cH_0 \oplus\cK,\qquad
\tau (\l)=\tau_0(\l)\oplus \wh \cK=\{\{h_0, \tau_0(\l)h_0\oplus k\}: h_0\in\cH_0,k\in\cK\}
\end{equation}
hold with subspaces $\cH_0=\cH\ominus\cK$, $\wh\cK=\{0\}\oplus\cK$  and a function  $\tau_0\in \wt R (\cH_0)$, whose values are densely defined operators. The operator function $\tau_0$ is called the operator part of $\tau$.

A  function $\tau\in \RH$ will be referred to the class $\RCH$  if  $\tau_0\in R_c [\cH_0]$. In the case  $\dim \cH<\infty$  one has $\RCH=\RH$.
\subsection{Meromorphic Nevanlinna functions and entire Nevanlinna pairs}
Below within this subsection we assume that $\cH$ is a Hilbert space with $\dim\cH<\infty$.

\begin{definition}\label{def2.6}
A  distribution $\xi:\bR\to\B(\cH)$ will be called discrete if
there exist sequences $F=\{t_k\}_{\nu_-}^{\nu_+}\in \cF$ and $\Xi=\{\xi_k\}_{\nu_-}^{\nu_+}$ with $0\leq\xi_k\in\B (\cH),\; \xi_k\neq 0,$ such that $\xi(t)$ is constant on each interval $(t_{k-1}, t_k)$ and $\xi(t_k+0)-\xi(t_k)=\xi_k$. For a discrete distribution $\xi$ defined by sequences $F$ and $\Xi$ we write $\xi=\{F,\Xi\}$
\end{definition}
\begin{definition}\label{def2.6.1}
An operator-valued function $\tau_0\in R[\cH_0]$ will be referred to the class $\Rm[\cH_0]$ if it admits a  continuation to an entire or meromorphic  function in $\bC$ (this function is denoted by $\tau_0$ as well). A relation-valued function $\tau \in \wt R(\cH)$ will be referred to the class $\RM (\cH)$ if its operator part $\tau_0$ (see \eqref{2.11}) belongs to $\Rm[\cH_0]$.
\end{definition}
For a function $\tau_0\in \Rm [\cH_0]$ we denote by $F_{\tau_0}$ the set of all poles of $\tau_0$. Clearly, $F_{\tau_0}$ admits the representation as a sequence $F_{\tau_0}=\{t_k\}_{\nu_-}^{\nu_+}\in \cF$.

As is known (see e.g. \cite[Proposition A.4.5]{BHS}) a function $\tau_0\in R [\cH_0]$ admits a continuation to  an interval
$(\a,\b)\subset \bR$ if and only if the spectral function $\xi$ of $\tau$ is constant on $(\a,\b)$. This fact and \eqref{2.9} yield the following assertion.
\begin{assertion}\label{ass2.6.2}
{\rm (i)} Let $\tau_0\in R [\cH_0]$. Then the following assertions are equivalent: {\rm (a)} $\tau_0\in\Rm [\cH_0]$; {\rm (b)} there is a set $F\in\cF$ such that $\tau_0$ admits a holomorphic continuation to $\bR\setminus F$; {\rm (c)} the  spectral function $\xi$ of $\tau$ is discrete.

{\rm (i)} Let $\tau_0 \in \Rm [\cH_0]$  and let $\xi=\{F, \Xi\}$ be a (discrete) spectral function of $\tau_0$ with $F=\{t_k\}_{\nu_-}^{\nu_+}$ and $\Xi=\{\xi_k\}_{\nu_-}^{\nu_+}$. Then $F=F_{\tau_0}$ and for any $t_k\in F_{\tau_0}$ there exists a function $\tau_k\in R [\cH_0]$ holomorphic at $t_k$ and  such that
\begin{gather}\label{2.13}
\tau_0(\l)=\tfrac 1 {t_k-\l}\,\xi_k+\tau_k(\l),\quad \l\in \bC\setminus F_{\tau_0}
\end{gather}
Equality \eqref{2.13} means that each $t_k\in F_{\tau_0}$ is a pole of $\tau_0$ of the first order and
\begin{gather}\label{2.15}
\xi_k= -\lim_{\l\to t_k}(\l-t_k)\tau_0 (\l)=-\underset{t_k}{\rm res}\,\tau_0,
\end{gather}
where $\underset{t_k}{\rm res}\,\tau_0$  denotes the residue of $\tau_0$ at  $t_k$.
\end{assertion}
In the following we deal with pairs $(C_0, C_1)$ of holomorphic operator functions $C_j:\cD\to \B (\cH), \; \cD\subset \bC, \; j\in\{0,1\}$. Clearly, the equality
\begin{gather*}
C(\l)=(C_0(\l),C_1(\l)):\cH\oplus\cH\to \cH, \quad \l\in\cD
\end{gather*}
enables one to identify such a pair with a holomorphic operator function $C:\cD\to \B (\cH^2,\cH)$.

Recall (see e.g. \cite{DM00,DM17}) that a pair $C=(C_0,C_1)$ of holomorphic operator functions $C_j:\CR\to \B (\cH),  \; j\in\{0,1\},$ is called a Nevanlinna pair in $\cH$ if for any $\l\in\CR$
\begin{gather}\label{2.18}
\im\l\cd\im (C_1(\l)C^*_0(\l)\geq 0,\quad C_1(\l)C_0^*(\ov\l)-C_0(\l)C_1^*(\ov\l)=0, \quad \ran C(\l)=\cH.
\end{gather}
Let  $J_\cH\in\B(\cH^2)$ be the operator given by
\begin{equation} \label{2.18.1}
J_\cH=\begin{pmatrix} 0 & -I_\cH \cr  I_\cH&
0\end{pmatrix}:\cH\oplus \cH \to \cH\oplus \cH.
\end{equation}
Then the conditions \eqref{2.18} can be written as
\begin{gather}\label{2.19}
i\, \im \l \cd C(\l)J_\cH C^*(\l)\leq 0, \quad C(\l)J_\cH C^*(\ov\l)=0, \quad \ran C(\l)=\cH, \;\; \l\in\CR.
\end{gather}
\begin{definition}\label{def2.6.4}
A pair $C=(C_0,C_1)$ of entire operator functions  $C_j:\bC\to \B (\cH)$ is said to be an entire Nevanlinna pair in $\cH$ if its restriction onto $\CR$ is a Nevanlinna pair and $\ran C(t)=\cH,\; t\in\bR$. The set of all entire Nevanlinna pairs in $\cH$ is denoted by $\EP$.
\end{definition}
Clearly, for a pair $C=(C_0, C_1)\in\EP$ the relations \eqref{2.18} (or equivalently \eqref{2.19}) hold for any $\l\in\bC$.

As is known (see e.g. \cite{DM00,DM17}) for each Nevanlinna pair $C=(C_0,C_1)$ in $\cH$ the equality
\begin {equation}\label{2.20}
\tau_C (\l):=\ker (C(\l))=\{\{h,h'\}\in
\cH^2: C_0(\l)h+C_1(\l)h'=0\},\quad \l\in\CR
\end{equation}
defines a relation-valued function $\tau_C\in\wt R (\cH)$.

Recall that an operator $X\in\B (\cH^2)$ is called $J_\cH$-unitary if $X^* J_{\cH} X=0$. If $X\in\B (\cH)$ is $J_\cH$-unitary, then $0\in\rho (X)$ and $X^{-1}$ is $J_{\cH}$-unitary as well.
\begin{lemma}\label{lem2.6.6}
Let $C=(C_0,C_1)$ be a Nevanlinna pair in $\cH$ and let $X\in\B (\cH^2)$ be a $J_\cH$-unitary operator. Then the equality $\wt C(\l)=C(\l)X^{-1},\; \l\in\CR,$ defines a Nevanlinna pair $\wt C= (\wt C_0, \wt C_1)$ in $\cH$ such that $\tau_{\wt C}(\l)= X \tau_C(\l)$. If in addition $C\in\EP$, then  the  equality $\wt C(\l)=C(\l)X^{-1},\; \l\in\bC,$ defines a pair $\wt C= (\wt C_0, \wt C_1)\in\EP$.
\end{lemma}
\begin{proof}
The required statements directly follows from conditions \eqref{2.19} which define a Nevanlinna pair.
\end{proof}
\begin{lemma}\label{lem2.6.7}
Let $C=(C_0,C_1)$ be a Nevanlinna pair in $\cH$, let $\tau=\tau_C\in \wt R(\cH)$ be given by \eqref{2.20} and let $\tau_0\in R[\cH_0]$ and $\cK$ be the operator and multivalued parts of $\tau$ respectively (see \eqref{2.11}).Then:

{\rm (i)} $\cK=\ker  C_1(\l), \; \l\in\CR$;

{\rm (ii)} if
\begin{gather}\label{2.21}
\begin{split}
C_0(\l)= (C_{00}(\l), \, C_{01}(\l)):\cH_0\oplus\cK\to \cH\qquad\qquad\\
 C_1(\l)= (C_{10}(\l), \, 0):\cH_0\oplus\cK\to \cH, \quad \l\in \CR
\end{split}
\end{gather}
are block representations of $C_j(\l)$ and $\wh C_1(\l)\in\B (\cH)$ is given by
\begin{gather}\label{2.22}
\wh C_1(\l)=(C_{10}(\l), -C_{01}(\l)):\cH_0\oplus\cK\to \cH,
\end{gather}
then $\ker \wh C_1(\l)=\{0\}$ and
\begin{gather}\label{2.23}
\ran (\wh C_1^{-1}(\l) C_{00}(\l) )\subset \cH_0, \qquad \tau_0(\l)=-\wh C_1^{-1}(\l) C_{00}(\l),\;\;\,\l\in\CR.
\end{gather}
\end{lemma}
\begin{proof}
(i) It follows from \eqref{2.20} that $\{0,h'\}\in\tau_C(\l) \Leftrightarrow h'\in\ker C_1(\l)$. This  yields statement (i).

(ii) One can easily verify that the equalities
\begin{gather}\label{2.24}
X=\begin{pmatrix} P_{\cH_0} & P_{\cK}\cr - P_{\cK} & P_{\cH_0} \end{pmatrix},\qquad  X^{-1}=\begin{pmatrix} P_{\cH_0} & -P_{\cK}\cr  P_{\cK} & P_{\cH_0} \end{pmatrix}
\end{gather}
defines a $J_{\cH}$-unitary operator $X\in\B (\cH\oplus \cH)$ and its inverse $X^{-1}$. Therefore by Lemma \ref{lem2.6.6} the equality $\wt C(\l)=C(\l)X^{-1}$ defines a Nevanlinna pair $\wt C= (\wt C_0, \wt C_1)$ in $\cH$ such that $\tau_{\wt C}(\l)= X \tau_C(\l)$. It follows from \eqref{2.21} and the second equality in \eqref{2.24} that
\begin{gather*}
\wt C_0(\l)= C_0(\l) P_{\cH_0}+C_1(\l) P_{\cK}=C_{00}(\l) P_{\cH_0}\\
\wt C_1(\l)= -C_0(\l) P_{\cK}+C_1(\l) P_{\cH_0}=-C_{01}(\l)P_{\cK}+ C_{10}(\l) P_{\cH_0}=\wh C_1(\l)
\end{gather*}
and, consequently,
\begin{gather}\label{2.26}
\tau_{\wt C}(\l)=\{\{h,h'\}\in\cH^2:C_{00}(\l) P_{\cH_0}h+\wh C_1(\l) h'=0\}, \quad \l\in\CR.
\end{gather}
On the other hand, \eqref{2.11} and the first equality in \eqref{2.24} yield $\tau_{\wt C}(\l)=\{\{h_0\oplus k, \tau_0(\l)h_0\}:h_0\in\cH_0, k\in\cK\}$. Hence
\begin{gather}\label{2.27}
\tau_{\wt C}(\l)=\tau_0(\l)P_{\cH_0}\in\B(\cH), \quad \l\in\CR.
\end{gather}
It follows from \eqref{2.26} that $\ker \wh C_1(\l)=\mul \tau_{\wt C}(\l)=\{0\}$ and consequently $\tau_{\wt C}(\l)=-\wh C_1^{-1}(\l) C_{00}(\l)P_{\cH_0}$. Combining this equality with \eqref{2.27} one gets \eqref{2.23}.
\end{proof}
\begin{proposition}\label{pr2.6.8}
Let $C=(C_0,C_1)\in \EP$. Then the equality \eqref{2.20} defines the relation-valued function $\tau_C\in\RM (\cH)$.
\end{proposition}
\begin{proof}
Since the restriction of $C$ onto $\CR$ is a Nevanlinna pair, it follows that $\tau_C\in \wt R (\cH)$. Let $\tau_0$ be the operator part of $\tau_C$. Then by   Lemma \ref{lem2.6.7} $\tau_0$  is defined  by the equality in \eqref{2.23} with  entire operator functions $\wh C_1$ and $C_{00}$. Let $\det (\l),\; \l\in\bC,$ be the determinant of the matrix of the operator $\wh C_1(\l)$ in some orthonormal basis of $\cH$ and let $Z=\{\l\in\bC: \det (\l)=0\}$. Since $\det$ is an entire function and $\det (\l)\neq 0,\; \l\in\CR,$ it follows that $Z\in\cF$ and $\ker \wh C_1(\l)=\{0\},\; \l\in\bC \setminus Z$. Therefore $\wh C_1^{-1}$ is holomorphic on the set $\bC\setminus Z$ and by inclusion in \eqref{2.23} $\ran (\wh C_1^{-1}(\l) C_{00}(\l) )\subset \cH_0,\; \l\in\bC\setminus Z $. Hence the equality in \eqref{2.23} with $\l\in \bC\setminus Z$ defines a holomorphic continuation of  $\tau_0$ onto $\bR\setminus Z$ and by Assertion \ref{ass2.6.2}, (i) $ \tau_0\in \Rm [\cH_0]$.
\end{proof}

\subsection{Boundary triplets and self-adjoint extensions}
In the following we denote by $A$ a closed symmetric linear relation (in particular closed not necessarily densely defined symmetric  operator) in a Hilbert space $\gH$.  Denote also by $n_\pm (A):=\dim \gN_\l(A^*), \;\l\in\bC_\pm,$ the deficiency indices of $A$, by $\ex $ the
set of all proper extensions of $A$ (i.e., the set of all
linear relations $\wt A$ in $\gH$ such that $A\subset\wt A\subset
A^*$) and  by $\cex$ the set of closed  extensions  $\wt A\in\ex $.

Below in this subsection we recall some definitions and results concerning boundary triplets for symmetric relations (see e.g. \cite{BHS,DM17,GorGor}).
\begin{definition}\label{def2.7}
A collection $\Pi=\bt$ consisting of a Hilbert space $\cH$ and linear mappings   $\G_j:A^*\to \cH, \; j\in\{0,1\},$  is called a
boundary triplet for $A^*$, if the mapping $\G=(\G_0,\G_1)^\top $ from $A^*$ into
$\cH\oplus\cH$ is surjective and the following  Green's
identity  holds:
\begin {equation*}
(f',g)-(f,g')=(\G_1  \wh f,\G_0 \wh g)- (\G_0 \wh f,\G_1
\wh g), \quad \wh f=\{f,f'\}, \; \wh g=\{g,g'\}\in A^*.
\end{equation*}
\end{definition}
\begin{remark}\label{rem2.7.0}
As is known $\mul A^*=\{0\}$  if and only if $\ov{\dom A}=\gH$. In this case $A$ and $A^*$ are densely defined operators and  operators $\G_j:A^*\to \cH$ in Definition \ref{def2.7} can be replaced with operators $\G_j:\dom A^*\to \cH, \; j\in\{0,1\}$.
\end{remark}
\begin{proposition}\label{pr2.7.1}  If $\Pi=\bt$ is a boundary triplet for $A^*$, then $n_+(A)=n_-(A)=\dim\cH$. Conversely, for each  symmetric relation $A\in\C (\gH)$  with   $n_+(A)=n_-(A)$ there exists a boundary triplet for $A^*$.
\end{proposition}
With a boundary triplet $\Pi=\bt$ for $A^*$ one associates an extension $A_0=A_0^*\in \cex$ given by
\begin{gather}\label{2.31}
A_0=\ker\G_0=\{\wh f\in A^*: \G_0\wh f=0\}.
\end{gather}
\begin{proposition}\label{pr2.8}
Let $\Pi=\bt$ be a boundary triplet for $A^*$ an let $A_0=A_0^*\in \cex$ be given by \eqref{2.31}. Then  the equality $\G_1\up\wh \gN_\l(A^*)=M(\l)\G_0\up\wh \gN_\l(A^*), \quad \l\in\rho (A_0),$ correctly defines the operator function $M:\rho (A_0)\to \B(\cH)$ belonging to the class $R_u [\cH]$.
\end{proposition}
The operator-functions  $M$ defined in Proposition \ref{pr2.8} is called the Weyl function of the triplet $\Pi$.

As is known a linear relation $\wt A=\wt A^*$ in a Hilbert space $\wt\gH\supset \gH$ is called an exit space extension of $A$ if $A\subset \wt A$ and there is no nontrivial subspace in $\wt \gH\ominus \gH$ reducing $\wt A$ (in the case $\ov{\dom A}=\gH$ each exit space extension $\wt A$ is an operator). Let $\wt A_j=\wt A_j^*\in \C (\wt\gH)$ be exit space extensions of $A$ and let $\gH_{rj}=\wt\gH_j\ominus \gH$,  so that $\wt\gH_j=\gH\oplus \gH_{rj},\; j\in\{1,2\}$. The extensions $\wt A_1$ and $\wt A_2$ are said to be equivalent if there  is a unitary operator $V\in\B (\gH_{r1},\gH_{r2})$ such that $\wt A_2=\wt U \wt A_1$ with the unitary operator $\wt U=(I_\gH\oplus V)\oplus (I_\gH\oplus V)\in \B(\wt\gH_1^2,\wt \gH_2^2)$. In the following we do not distinguish equivalent extensions.

An exit space extension $\wt A=\wt A^* \in \C (\wt\gH)$ of $A$ is called canonical if $\wt\gH=\gH$ (i.e., if $\wt A\in\cex$).

A description of all exit space  (in particular canonical) self-adjoint extensions $\wt A$ of $A$ in terms of the boundary triplet is given by the following theorem.
\begin{theorem}\label{th2.10}
Assume that $\Pi=\bt$ is a boundary triplet for  $A^*$. Then:

{\rm (i)} The
mapping
\begin {equation}\label{2.32}
\t\to  A_\t :=\{ \hat f\in A^*:\{\G_0  \hat f,\G_1 \hat f \}\in
\t\}
\end{equation}
establishes a bijective correspondence $\wt A=A_\t$ between all linear
relations  $\t$ in $\cH$ and all extensions $ \wt A\in\ex$. Moreover, the following holds: {\rm (a)} $A_\t\in\cex$ if and only if $\t\in\C (\cH)$;  {\rm (b)} $A_\t$ is  symmetric (self-adjoint) if  and only if $\t$ is symmetric (resp. self-adjoint). In this case $n_\pm (A)=n_\pm (\t)$.

{\rm (ii)} The equality
\begin{gather}\label{2.34}
P_\gH (\wt A_\tau -\l)^{-1}\up\gH =(A_{-\tau(\l)}-\l)^{-1}, \quad \l\in\CR
\end{gather}
establishes a bijective correspondence $\wt A=\wt A_\tau$ between all relation valued functions $\tau=\tau(\l)\in \RH$ and all exit space self-adjoint extensions $\wt A$ of $A$. Moreover, an extension $\wt A_\tau$ is canonical  if and only if $\tau(\l)\equiv \t(=\t^*), \; \l\in\CR$, in which case $\wt A_\tau=A_{-\t}$.
\end{theorem}
\begin{proposition}\label{pr2.10.1}
Let $\Pi=\bt $ be a boundary triplet for $A^*$ and  let $X\in\B (\cH^2)$ be a $J_\cH$-unitary operator with $J_\cH$ of the form \eqref{2.18.1}. Then:

{\rm(i)}  $\wt\Pi:=\{\cH,\wt\G_0,\wt\G_1\} $ with  $(\wt\G_0,\wt \G_1)^\top=X(\G_0, \G_1)^\top $ is a boundary triplet for $A^*$;

{\rm(ii)} if $\t\in\C (\cH)$ and $\wt A=A_\t\in\cex $ (in the triplet $\Pi$), then $\wt A=A_{\wt\t}$ (in the triplet  $\wt\Pi$) with $\wt\t=X\t$;

{\rm(iii)} if $\tau\in \wt R(\cH)$ and $\wt A=\wt A_\tau$ is an exit space extension of $A$ (in the triplet $\Pi$), then $\wt A=\wt A_{\wt\tau}$ (in the triplet $\wt\Pi$) with $\wt\tau\in \wt R(\cH)$ given by $-\wt\tau(\l)=X(-\tau(\l)), \; \l\in\CR$.
\end{proposition}
\begin{remark}\label{rem2.11}
(i) Let $\Pi=\bt $ be a boundary triplet for $A^*$, $\t\in\C (\cH)$ and $\wt A=A_{-\t}\in\cex$.  Consider the abstract boundary value problem
\begin{gather}\label{2.36}
\{f,\l f\}\in A^*, \quad \{\G_0\{f,\l f\},\G_0\{f,\l f\} \}\in -\t,
\end{gather}
where the second relation  is the abstract boundary condition depending on the  boundary parameter $\l\in\bC$. It is clear that for a given $\l$ the  set of all solutions of \eqref{2.36} (i.e., the set of all $f\in\dom A^*$ satisfying \eqref{2.36}) coincides with $\gN_\l(\wt A)$. Hence $\s_p(\wt A)$ coincides with the set of all $\l\in\bC$ for which the set of all solutions of \eqref{2.36} is nontrivial.

(ii) The same parametrization $\wt A=\wt A_\tau$  as in Theorem \ref{th2.10}, (ii) can be  given by means of the Krein formula for generalized resolvents
\begin{gather}\label{2.37}
P_\gH (\wt A_\tau -\l)^{-1}\up\gH = (A_0-\l)^{-1}- \g(\l)(\tau (\l)+M(\l))^{-1} \g^*(\ov\l), \quad \l\in\CR,
\end{gather}
where $A_0=A_0^*\in\cex$ is a (basic) extension \eqref{2.31}, $M(\l)$ is the Weyl function of the triplet $\Pi$ and $\g(\l)$ is a so-called $\g$-field (for more details see \cite{LanTex77,Mal92}).
\end{remark}
\section{Shtraus family and abstract eigenvector expansion}
\subsection{Shtraus family of linear relations}
Assume that $\wt\gH\supset\gH$ is a Hilbert space, $\gH_r:=\wt\gH\ominus\gH $,  $P_\gH$ is the orthoprojection in $\wt\gH$ onto $\gH$ and $\wt A=\wt A^*\in\C (\wt \gH)$ is an exit space extension of $A$. Recall that a linear relation $C_{\wt A}$ in $\gH$ given by
\begin{gather}\label{3.1}
C_{\wt A}:=P_{\gH}\wt A\up \gH=\{{\{f,f'\}\in \gH^2: \{f,f'\oplus f_r'\}\in \wt A}\;\;\text{\rm with some}\;\; f_r'\in\gH_r\}
\end{gather}
is called the compression of $\wt A$. It is easy to see that $C_{\wt A}$ is a (not necessarily closed) symmetric extension of $A$. Note also that the equality
\begin{gather}\label{3.2}
T_{\wt A}:=\{\{P_\gH f, P_\gH f'\}:\{f,f'\}\in \wt A \}
\end{gather}
defines a linear relation $T_{\wt A}$ in $\gH$ and $A\subset C_{\wt A}\subset T_{\wt A}\subset A^*$.
\begin{definition}\label{def3.1}$\,$\cite{BHS,DijLan20,Sht70} A family of linear relations $\SA (\l)$ in $\gH$ defined   by
\begin{gather}\label{3.3}
\SA (\l)=\{\{f,f'\}\in\gH^2: \{f\oplus f_r, f'\oplus\l f_r\}\in\wt A \; \text{ with some } \; f_r\in\gH_r\}, \quad \l\in\bC
\end{gather}
is called the Shtraus family of $\wt  A$.
\end{definition}
It is easy to see that $\SA (\l)=(P_\gH(\wt A-\l)^{-1}\up\gH  )^{-1}+\l $, which is equivalent to \begin{gather}\label{3.3.1}
(\SA (\l)-\l)^{-1}=P_\gH(\wt A-\l)^{-1}\up\gH , \quad \l\in\bC.
\end{gather}
Moreover, $\SA (t)$ is a symmetric extension of $A$ and $A\subset \SA (t)\subset T_{\wt A}\subset A^*$ for any $t\in\bR$.
\begin{lemma}\label{lem3.1.1}
Let $\Pi=\bt$ be a boundary triplet for $A^*$. Then for any $t\in\bR$ the following statements hold:

{\rm (i)} $A^t:=(A-t)^{-1}$ is a closed symmetric relation in $\gH$, $(A^t)^*=(A^*-t)^{-1}$ and the collection $\Pi_t=\{\cH,\G_{0t}, \G_{1t}\}$ with
\begin{gather}\label{3.4}
\G_{0t}\{f,f'\}=\G_0\{f',f+tf'\}, \qquad \G_{1t}\{f,f'\}= -\G_1\{f',f+tf'\},\quad \{f,f'\}\in (A^t)^*
\end{gather}
is a boundary triplet for  $(A^t)^*$.

{\rm (ii)} If $\wt A\in \ex$, then $ (\wt A-t)^{-1}\in {\rm ext} (A^t)$ and $\G_t (\wt A-t)^{-1}= -\G\wt A$ (here $\G=(\G_0,\G_1)^\top$ and  $\G_t=(\G_{0t},\G_{1t})^\top$).

{\rm (iii)} If $\tau\in\wt R (\cH)$ and $\wt A=\wt A_\tau$ is the exit space extension of $A$, then $\wt A^t:=(\wt A-t)^{-1}$ is an exit space extension of $A^t$ and $\wt A^t=\wt A^t_{\tau^t}$ (in the triplet $\Pi_t$) with $\tau^t\in \wt R(\cH)$ given by
\begin{gather}\label{3.5}
\tau^t=\tau^t(\l)=-\tau \left(t+\tfrac 1 \l\right),\quad \l\in\CR.
\end{gather}
\end{lemma}
\begin{proof}
Statement {\rm (i)} and {\rm (ii)} are obvious.

(iii) Let $\tau\in\wt R(\cH)$ and $\wt A=\wt A_\tau$. We show that \begin{gather}\label{3.6}
S_{\wt A^t}(\l)=\left(\SA \left(t+ \tfrac 1 \l \right)-t\right)^{-1}, \quad \l\in\CR.
\end{gather}
Let $\l\in\CR$. Then by \eqref{3.3} $\{f,f'\}\in S_{\wt A^t}(\l)$ if and only if $\{f\oplus f_r, f'\oplus\l f_r\}\in\wt A^t$  with some $f_r\in\gH_r$. Now  the equivalences
\begin{gather*}
\{f\oplus f_r, f'\oplus\l f_r\}\in\wt A^t\iff \{f'\oplus \l f_r, (f+tf')\oplus (1+\l t) f_r\}\in\wt A\iff\\
\{f',f+tf'\}\in \SA \left(t+\tfrac 1 \l \right)\iff \{f,f'\}\in \left( \SA\left( t+\tfrac 1 \l \right)-t  \right)^{-1},
\end{gather*}
proves \eqref{3.6}. Next, by \eqref{2.34} and \eqref{3.3.1} the equality $\wt A=\wt A_\tau$ is equivalent to $\SA (\l)=A_{-\tau(\l)}$, that is $\G \SA(\l)=-\tau(\l), \; \l\in\CR$. Hence by \eqref{3.6} and statement (ii)
\begin{gather*}
\G_t S_{\wt A^t}(\l)=-\G \SA\left(t+\tfrac 1 \l \right)=\tau \left(t+\tfrac 1 \l \right)=-\tau ^t(\l),\quad \l\in\CR,
\end{gather*}
which implies that $\wt A^t=\wt A^t_{\tau^t}$.
\end{proof}
In the following theorem we characterize in terms of the parameter $\tau$ the Shtraus family $\SA (t )$, $t\in\bR$, corresponding to the exit space extension $\wt A=\wt A_\tau$ of $A$.
\begin{theorem}\label{th3.2}
Assume that $\Pi=\bt$ is a boundary triplet for $A^*$, $\tau\in\RCH$, $\wt A=\wt A_\tau$ is the corresponding exit space self-adjoint extension of $A$ and $\SA (\l)$ is the Shtraus family of $\wt A$. Moreover, let $\tau_0\in R_c[\cH_0]$ and $\cK$ be the operator and multivalued parts of $\tau$ respectively (see \eqref{2.11}), let $t\in\bR$ and let  $\cB_{\tau_0}(t)\in\B (\cH_0)$ and $D_{\tau_0}(t):\dom D_{\tau_0}(t)\to \cH_0 \;(\dom D_{\tau_0}(t)\subset\cH_0)$ be operators corresponding to $\tau_0$ in accordance with Proposition \ref{pr2.4}, (ii). Assume also that $\ran \cB_{\tau_0}(t)$ is closed. Then $\SA (t)=A_{\eta(t)}=\{\wh f\in A^*:\{\G_0\wh f,\G_1 \wh f\}\in\eta(t)\}$, where $\eta(t)=\eta_\tau(t)$ is given by
\begin{gather}\label{3.10}
\eta_\tau(t)=\{\{h,-D_{\tau_0}(t)h+\cB_{\tau_0}(t)h_0+\kappa \}: h\in\dom D_{\tau_0}(t), h_0\in  \cH_0, \kappa\in\cK\}.
\end{gather}
\end{theorem}
\begin{proof}
Let $\tau^t\in \RH$ be given by \eqref{3.5}. Then according to Lemma \ref{lem3.1.1} $A^t=(A-t)^{-1}$ is a symmetric relation in $\gH$, $\wt A^t =(\wt A-t)^{-1}$ is an exit space self-adjoint extension of $A^t$ and $\wt A^t=\wt A_{ \tau_t}^t$ in the triplet $\Pi_t=\{\cH,\G_{0t},\G_{1t}\}$ for $(A^t)^* $ given by \eqref{3.4}. It is clear that the multivalued part of $\tau^t$ coincides with $\cK$, while the operator part $\tau_0^t$ of $\tau^t$ is
\begin{gather}\label{3.11}
\tau_0^t(\l)=-\tau_0 \left(t+\tfrac 1 \l\right),\quad \l\in\CR.
\end{gather}
Hence $\tau_0^t\in R_c[\cH_0]$ and, consequently, $\tau^t\in \RCH$. Moreover, by \eqref{3.11}
\begin{gather}\label{3.11.1}
\dom D_{\tau_0^t \infty}=\dom D_{\tau_0}(t), \quad D_{\tau_0^t \infty}=-D_{\tau_0}(t),\quad  \cB_{\tau_0^t\infty}=\cB_{\tau_0}(t),
\end{gather}
where $\cB_{\tau_0^t\infty}\in\B (\cH_0)$ and $D_{\tau_0^t \infty}:\dom D_{\tau_0^t \infty}\to\cH_0\; (\dom D_{\tau_0^t \infty}\subset\cH_0)$ are the operators corresponding to $\tau_0^t $ in accordance with Proposition \ref{pr2.4}, (i). Hence $\ran \cB_{\tau_0^t\infty}$ is closed.

Let $\G_t=(\G_{0t}, \G_{1t})^\top$ and let $C_{\wt A^t}$ be the compression of $\wt A^t$. Then by \cite[Theorem 3.8]{Mog19} $\G_t C_{\wt A^t}=-\eta (t)$, where
\begin{gather}\label{3.12}
\eta (t)=\{\{h, D_{\tau_0^t \infty}h+\cB_{\tau_0^t\infty}h_0+k\}: h\in\dom D_{\tau_0^t \infty}, h_0\in\cH_0, k\in\cK\}.
\end{gather}
It follows from \eqref{3.3.1} that $(\SA(t)-t)^{-1}=C_{\wt A^t}$ and by Lemma \ref{lem3.1.1}, (ii) $\G \SA (t)=-\G_t C_{\wt A^t}=\eta(t)$. Hence $\SA(t)=A_{\eta(t)}$ with $\eta(t)=\eta_\tau(t)$ given by \eqref{3.12} and \eqref{3.11.1} yields \eqref{3.10}.
\end{proof}
Recall that two extensions  $\wt A_1,\wt A_2 \in\cex$ are called transversal if $\wt A_1\cap \wt A_2=A$ and $\wt A_1\wh + \wt A_2:= \{\wh f+\wh g: \wh f\in\wt A_1,\, \wh g\in\wt A_2\}=A^*$.

Let as before $\Pi=\bt$ be a boundary triplet for $A^*$ and let $A_0(=\ker\G_0)$ be a basic self-adjoint extension of $A$ in the Krein formula \eqref{2.37}. In the following two Theorems \ref{th3.4} and \ref{th3.5} we characterise in terms of the Nevanlinna parameter $\tau$ exit space extensions $\wt A=\wt A_\tau$ of $A$ such that the Shtraus family $\SA (t)$ satisfies one of the following extremal conditions: (i) $\SA (t)\subset A_0$ (in particular, $\SA (t)= A_0$ ); (ii) $\SA (t)$ is self-adjoint and transversal with $A_0$. For the compression $\CAt$ of $\wt A_\tau$ similar results were obtained in our paper \cite[Theorems 3.9 and 3.16]{Mog19}. Similarly to Theorem \ref{th3.2} the proof of Theorems \ref{th3.4} and \ref{th3.5}
can be easily obtained by application of results from \cite{Mog19} to the symmetric relation $A^t=(A-t)^{-1}$ and the boundary triplet $\Pi_t=\{\cH,\G_{0t},\G_{1t}\}$ for $(A^t)^*$ (see \eqref{3.4}). Therefore we omit the proof.
\begin{theorem}\label{th3.4}
Assume that $\Pi=\bt$ is a boundary triplet for $A^*$,  $A_0=\ker \G_0$, $\tau \in \wt R(\cH)$,  $\wt A=\wt A_\tau$ is the exit space extension of $A$ and $t$ is a real point. Then:

{\rm (i)} The following statements {\rm(a1)} and  {\rm(a2)} are equivalent:

{\rm (a1)} $T_{\wt A}$ is closed and $\SA (t)\subset A_0$ (for $T_{\wt A}$ see \eqref{3.2});

{\rm (a2)} $\tau \in \RCH$ and the operator part  $\tau_0\in R_c[\cH_0]$ of $\tau$ satisfies
\begin{gather}\label{3.16}
\lim_{y\to 0}\tfrac 1 y \im (\tau_0(t+iy)h,h)=\infty, \quad h\in\cH_0, \;\; h\neq 0.
\end{gather}
Moreover, if statement {\rm (a1)} (equivalently {\rm (a2)} ) is valid and $\ran \cB_{\tau_0}(t)$ is closed, then $\SA(t)$ is closed and $\SA(t)=\{\wh f\in A^*: \G_0\wh f=0, \; \G_1\wh f\in \ran \cB_{\tau_0}(t)\oplus\cK\}$ (for $\cK$ see \eqref{2.11}).
\vskip 2mm
{\rm (ii)} The following statements ${\rm(a1')}$ and  ${\rm(a2')}$ are equivalent:
\vskip 2mm
{\rm (a1$^\prime$)} $T_{\wt A}$ is closed and $\ov{\SA (t)}= A_0$; \quad {\rm (a2$^\prime$)} $\tau \in \RCH$ and $\ker\cB_{\tau_0}(t)=\{0\}$.

Moreover, if statement {\rm (a1$^\prime$)} (equivalently {\rm (a2$^\prime$)} ) is valid and $\ran \cB_{\tau_0}(t)$ is closed, then $\SA(t)=A_0$.
\vskip 2mm
{\rm (iii)} $T_{\wt A}$ is closed  and $\SA (t)= A$ if and only if $\tau\in R_c[\cH],\; \cB_\tau(t)=0$  and \eqref{3.16} holds with $\tau$ in place of $\tau_0$.
\end{theorem}
\begin{theorem}\label{th3.5}
Let the assumptions be the same as in Theorem \ref{th3.4}. Then:

{\rm (i)} If $\tau \in\RCH$, $\ran\cB_{\tau_0}(t)$ is closed and
\begin{gather}\label{3.17}
\lim_{y\to 0}\tfrac 1  y\im (\tau_0(t+iy)h,h)<\infty, \quad h\in \ker\cB_{\tau_0}(t),
\end{gather}
then $\SA(t)$ is self-adjoint (in \eqref{3.17} $\tau_0\in R_c[\cH_0]$ is the operator part of $\tau$).
\vskip 2mm
{\rm (ii)} The following statements {\rm(b1)} and  {\rm(b2)} are equivalent:

{\rm (b1)} $T_{\wt A}$ is closed, $\SA(t)$ is self-adjoint  and transversal with $A_0$;

{\rm (b2)} $\tau\in R_c[\cH]$ and
\begin{gather}\label{3.18}
\lim_{y\to 0} \tfrac 1 y\im (\tau(t+iy)h,h)<\infty, \quad h\in \cH.
\end{gather}
Moreover, if {\rm (b1)}  (equivalently {\rm (b2)}) is satisfied, then
\begin{gather}\label{3.18.1}
\SA(t)=A_{-K}=\{\wh f\in A^*: \G_1\wh f = - K \G_0\wh f \},
\end{gather}
with the operator $K=K^*\in\B(\cH)$ given by
\begin{gather}\label{3.19}
K=s \text{-} \lim_{y\to 0} \tau(t+iy).
\end{gather}
\end{theorem}
\subsection{The case of finite deficiency indices}
\begin{theorem}\label{th3.6}
Assume that:

{\rm (i)} $A\in\C (\gH)$ is a symmetric relation with equal finite deficiency indices $n_+(A)=n_-(A)<\infty$,  $\Pi=\bt$ is a boundary triplet for $A^*$, $\tau\in\RH$ and  $\wt A=\wt A_\tau$ is  the corresponding exit space self-adjoint extension of $A$;

{\rm (ii)}  $\tau_0\in R[\cH_0]$ and $\cK$ are the operator and multivalued parts  of $\tau$ respectively (see \eqref{2.11}).

Then  the Shtraus family $\SA (t)$ of $\wt A$ is given in the triplet $\Pi$ by
\begin{gather}\label{3.20}
\SA (t)=A_{\eta(t)}=\{\wh f\in A^*:\{\G_0\wh f,\G_1 \wh f\}\in\eta(t)\},\quad t\in\bR,
\end{gather}
where $\eta=\eta_\tau:\bR\to \C (\cH)$ is the relation-valued  function \eqref{3.10}.
\end{theorem}
\begin{proof}
Since by Proposition \ref{pr2.7.1} $\dim \cH<\infty$, it follows that  $\wt R(\cH)=\wt R_c(\cH)$ and  $\ran\cB_{\tau_0}(t)$ is closed. Now the required statement is implied by Theorem \ref{th3.2}.
\end{proof}
\begin{remark}\label{rem3.6.1}
Theorems \ref{th3.4}, \ref{th3.5} and \ref{th3.6} readily yield the results obtained in \cite{DijLan20} for the Shtraus family of the extension $\wt A\supset A$ in the case $n_\pm(A)<\infty$.
\end{remark}

\begin{corollary}\label{cor3.7}
Let under the assumption {\rm (i)} of Theorem \ref{th3.6} $\tau\in R[\cH]$ and $\tau$ admits a holomorphic continuation at the point $t_0\in\bR$. Then $S_{\wt A}(t_0)=A_{-\tau(t_0)}$ (in the triplet $\Pi$)
\end{corollary}
\begin{proof}
The equality $\im \tau(t_0)=0$ yields
\begin{gather*}
\lim_{y\to 0}\tfrac 1 y \im (\tau(t_0+iy)h,h) = {\rm Re}\, (\tau'(t_0)h,h)<\infty, \quad h\in\cH.
\end{gather*}
Hence $\dom D_\tau(t_0)=\cH$ and $D_\tau(t_0)=\tau(t_0)$. Moreover, $\cB_\tau(t_0)=0$ and by \eqref{3.10}
\begin{gather}\label{3.21}
\eta(t_0)=\eta_\tau(t_0)=-\tau(t_0).
\end{gather}
This and Theorem \ref{th3.6} yield the result.
\end{proof}

In the following proposition we characterize in terms of abstract boundary conditions the Shtraus family of the exit space extension $\wt A_\tau$ with the parameter $\tau\in \wt R(\cH)$ generated by an entire  Nevanlinna pair.
\begin{proposition}\label{pr3.8}
Let under the assumption {\rm (i)} of Theorem \ref{th3.6} $\tau=\tau_C\in \RM (\cH)$ be a relation-valued function defined by \eqref{2.20} with an entire Nevanlinna pair $C=(C_0,C_1)\in\EP$. Then the equality  (the abstract boundary condition)
\begin{gather}\label{3.27}
S_{\wt A_\tau}(t)=\{\wh f\in A^*:C_0(t)\G_0 \wh f-C_1(t)\G_1 \wh f=0\},\quad t\in\bR
\end{gather}
defines the Shtraus family $S_{\wt A_\tau}(t)$ of $\wt A_\tau$.
\end{proposition}
\begin{proof}
Assume that $t\in\bR$. Since $\im C_1(t)C_0^*(t)=0$ and $\ran (C_0(t), C_1(t))=\cH$, it follows from \cite[Proposition 6.46]{DM17} that the equality
\begin{gather*}
\t=\ker C(t)= \{\{h,h'\}\in\cH^2:C_0(t)h+ C_1(t) h'=0 \}
\end{gather*}
defines a relation $\t=\t^*\in\C (\cH)$. Hence by \eqref{2.7}
\begin{gather}\label{3.28}
\cH=\cH_0\oplus\cK, \qquad \t=\{\{h_0,\t_{\rm op}h_0\oplus k\}: h_0\in\cH_0, k\in\cK\},
\end{gather}
where $\cK=\mul\t$  and $\t_{\rm op}=\t_{\rm op}^*\in\B (\cH_0)$ is the operator part of $\t$. It was shown in the proof of Lemma \ref{lem2.6.7} that the equality  \eqref{2.24}  defines a $J_\cH$-unitary operator $X\in\B (\cH^2)$. Therefore by Lemma \ref{lem2.6.6}  the equality $\wt C(\l)=C(\l)X^{-1},\; \l\in\bC,$ defines a  pair $\wt C= (\wt C_0, \wt C_1)\in\EP$ such that $\tau_{\wt C}(\l)= X \tau_C(\l),\;\l\in\CR$. Next we put $\wt \t := X\t$. Then
\begin{gather}\label{3.29}
\wt\t=\ker\wt C(t)=\{\{h,h'\}\in\cH^2: \wt C_0(t)h+ \wt C_1(t) h'=0 \}
\end{gather}
and, consequently, $\mul \wt\t=\ker \wt C_1(t)$.  On the other hand, by \eqref{3.28} $\wt\t=\t_{\rm op}P_{\cH_0}\in\B (\cH)$. Thus $\ker \wt C_1(t)=\{0\}$ and, consequently, there is a neighbourhood $U(t)$ of $t$ in $\bC$ such that $\ker\wt C_1(\l)=\{0\}, \; \l\in U(t)$. This and formula \eqref{2.20} (for $\wt C$) imply that $\mul \tau_{\wt C}(\l)=\{0\},\; \l\in U(t)\setminus \bR,$ and, consequently, the multivalued part $\wt\cK$ of $\tau_{\wt C}\in \wt R(\cH)$ is $\wt\cK  =\mul \tau_{\wt C}(\l)=\{0\},\; \l\in\CR$ (this means that  $\tau_{\wt C}\in \wt R [\cH]$). Therefore  $\ker  \wt C_1(\l)=\{0\}$ and  $\tau_{\wt C}(\l)=-\wt C_1^{-1}(\l)\wt C_0(\l),\; \l\in\CR$. Moreover, $\tau_{\wt C} $ admits a holomorphic continuation at the point $t$ and $\tau_{\wt C}(t)=-\wt C_1^{-1}(t)\wt C_0(t)=\wt \t\in \B(\cH)$.

Let $\wt J={\rm diag}\, (I_{\cH}, -I_{\cH})\in\B (\cH^2)$. Then $\wt X:=\wt J X \wt J\in B (\cH^2)$ is a $J$-unitary operator  and by Proposition \ref{pr2.10.1}, (i) the equality $(\wt \G_0,\wt \G_1)^\top= \wt X (\G_0, \G_1)^\top $ defines a boundary triplet $\wt\Pi=\{\cH, \wt\G_0,\wt \G_1\}$ for $A^*$. Moreover, $\wt X (-\tau_C(\l))=-\tau_{\wt C}(\l),\; \l\in\CR,$ and by Proposition \ref{pr2.10.1}, (iii) $\wt A=\wt A_{\tau_{\wt C}}$ (in the triplet $\wt \Pi$). Therefore by Corollary \ref{cor3.7} $S_{\wt A}(t)=A_{-\tau_{\wt C}(t)}=A_{-\wt\t}$ (in the triplet $\wt\Pi$). Finally, the equality $-\wt \t=\wt X (-\t)$ and Proposition \ref{pr2.10.1}, (ii) yield $S_{\wt A}(t)=A_{-\t}$ (in the triplet $\Pi$), which is equivalent to \eqref{3.27}.
\end{proof}

\subsection{Abstract eigenvector expansion}
Assume that $A\in\C (\gH)$ is a symmetric relation with finite deficiency indices $n_+(A)=n_-(A)=:d$ and $\Pi=\bt$ is a boundary triplet for $A^*$ (hence by Proposition \ref{pr2.7.1} $\dim \cH=d<\infty$).  Moreover, let $\tau\in \RH$ and let $\eta(t)=\eta_\tau(t)$ be a $\C (\cH)$-valued function given for any $t\in\bR$ by \eqref{3.10}. We consider the abstract eigenvalue problem
\begin{gather}
\{f,tf\}\in A^*\label{4.1}\\
\{\G_0\{f,tf\},\G_1\{f,tf\}\}\in \eta_\tau(t)\label{4.2}
\end{gather}
with the abstract boundary condition \eqref{4.2} depending on the parameter $t\in\bR$. The set of all solutions $f\in\dom A^*$  of the problem \eqref{4.1}, \eqref{4.2} will be denoted by $\nt$.  Clearly, $\nt$ is a linear manifold in $\gN_t(A^*)$.
\begin{definition}\label{def4.2}
A  point $t\in\bR$ is called an eigenvalue of the problem \eqref{4.1}, \eqref{4.2} if $\nt\neq\emptyset$. The set of all such eigenvalues is denoted by $\wt {EV}$. An element $f\in \nt$ for $t\in  \wt {EV}$ is called an eigenvector of the problem  \eqref{4.1}, \eqref{4.2} corresponding to $t$.
\end{definition}
\begin{remark}\label{rem4.3}
In the case of a densely defined operator $A$ the operators $\G_0$ and $\G_1$ of the triplet $\Pi=\bt$ are defined on $\dom A^*$ (see Remark \ref{rem2.7.0}) and the eigenvalue problem \eqref{4.1}, \eqref{4.2} takes the form
\begin{gather}
 A^*f=tf\label{4.3}\\
\{\G_0 f,\G_1 f \}\in \eta_\tau(t)\label{4.4}.
\end{gather}
Assume now that $C=(C_0,C_1)$ is an  entire Nevanlinna pair in $\cH$ and $\tau=\tau_C\in\wt R(\cH)$ is given by \eqref{2.20}. Then by \eqref{3.20} and \eqref{3.27} $\eta_\tau(t)=\{\{h,h'\}\in\cH^2:C_0(t)h- C_1(t)h'=0 \}$ and the boundary condition \eqref{4.4} can be written as
\begin{gather}\label{4.4.1}
C_0(t)\G_0 f -C_1(t)\G_1 f=0.
\end{gather}
\end{remark}
\begin{proposition}\label{pr4.4}
Let under the assumptions of Theorem \ref{th3.6} $\wt A$ be a linear relation in a Hilbert space $\wt\gH\supset \gH$. Moreover, let $\eta_\tau(t)$ be given by \eqref{3.10}, let $\nt\; (t\in\bR)$ be the set of all solutions of the problem \eqref{4.1}, \eqref{4.2} and let $\wt {EV}$ be the set of all eigenfunctions od the same problem. Then
\begin{gather}\label{4.5}
P_\gH \gN_t(\wt A)=\nt, \;\;t\in\bR,\quad {\rm and} \quad \s_p(\wt A)=\wt {EV}.
\end{gather}
\end{proposition}
\begin{proof}
Let $\gH_r:=\wt\gH\ominus\gH$ and let $t\in\bR$. Then $f\in P_\gH\gN_t(\wt A)$ if and only if $f\oplus f_r \in \gN_t(\wt A)$ or, equivalently, $\{f\oplus f_r, tf\oplus tf_r\}\in \wt A$ with some $f_r\in\gH_r$. Therefore by \eqref{3.3} $P_{\gH}\gN_t(\wt A)=\{f\in\gH:\{f,tf\}\in \SA (t)\}=\gN_t(\SA (t))$. Moreover, by Theorem \ref{th3.6} $\SA(t)=A_{\eta_\tau (t)}$ and Remark \ref{rem2.11}, (i) implies that $\gN_t(\SA (t)) =\nt$. This yields the first equality in \eqref{4.5}.

Next, in view of \eqref{2.7} one has $\wt\gH=\wt\gH_0\oplus \mul\wt A$ and $\wt A={\rm gr}\, \wt A_0\oplus \wh \mul \wt A$,  where $\wt A_0=\wt A_0^*$ is an operator in $\wt\gH_0$. Let $t\in\bR$.  Since $\gN_t(\wt A)=\gN_t(\wt A_0)$, it follows that $\gN_t(\wt A)\cap\gH_r \subset \gN_t(\wt A_0)$. Therefore  the subspace $\gN_t(\wt A)\cap\gH_r$ reduces $\wt A_0$ and, consequently, $\wt A$. Hence $\gN_t(\wt A)\cap\gH_r=\{0\}$ and therefore $\ker (P_\gH\up \gN_t(\wt A))=\{0\}$. Thus by the first equality in \eqref{4.5} $\nt=\{0\}\Leftrightarrow\gN_t(\wt A)=\{0\}$, which yields the second equality in \eqref{4.5}.
\end{proof}
\begin{remark}\label{rem4.4.1}
Clearly in the case $\tau(\l)\equiv \t(=\t^*)$ (i.e., in  the case of the canonical extension $\wt A=A_{-\t}$) Proposition \ref{pr4.4} turns into the statements of Remark \ref{rem2.11}, (i).
\end{remark}
\begin{lemma}\label{lem4.5}
Assume that $T,\wt T\in\C (\gH)$ and $T\subset \wt T$. Then:

{\rm (i)} If $\dim\ker \wt T<\infty$ and $\ran \wt T$ is closed, then the same statements hold for $T$;

{\rm (ii)} If $\dim\ker  T<\infty$,  $\ran  T$ is closed and $\dim \wt T/T<\infty$, then  $\dim\ker \wt T<\infty$ and $\ran \wt T$ is closed.
\end{lemma}
\begin{proof}
(i) Assume that $\dim\ker \wt T<\infty$ and $\ran \wt T$ is closed. Since $\ker T\subset \ker \wt T$, it follows that $\dim\ker T<\infty$. Next, $T$ and $\wt T$ admit the representations
\begin{gather}\label{4.7}
T=T_1\oplus \wh \gN_0(T), \qquad \wt T=\wt T_1\oplus \wh \gN_0(T),
\end{gather}
where $T_1, \wt T_1\in\C (\gH), \; T_1\subset \wt T_1$ and $\ker T_1=\{0\}$. Since $\dim \wh\gN_0(\wt T_1)<\infty$, it follows that $T_1\dotplus \wh\gN_0(\wt T_1)$ is a   closed subspace in $\wt T_1$ and hence $\wt T_1$ admits the representation
\begin{gather}\label{4.8}
\wt T_1=T_0\oplus(T_1\dotplus \wh\gN_0(\wt T_1))
\end{gather}
with some $T_0\in\C(\gH)$. Let $\wt T_2:=T_0\oplus T_1$. Then $\wt T_2 \in \C (\gH), \; T_1\subset \wt T_2\subset \wt T_1$ and by \eqref{4.8} $\wt T_2\cap\wh \gN_0(\wt T_1)=\{0\}$, which implies that $\ker \wt T_2=\{0\}$. Moreover, by \eqref{4.8} $\ran \wt T_2=\ran \wt T_1$ and the second equality in \eqref{4.7} yields $\ran \wt T_1=\ran \wt T$. Hence $\ran \wt T_2$ is closed and consequently $\wt T_2^{-1}$ is a bounded operator from $\ran \wt T_2$ into $\gH$. Since $T_1^{-1}\subset \wt T_2^{-1}$ and $T_1^{-1}\in\C (\gH)$, it follows that $\ran T_1(=\dom T_1^{-1})$ is closed. Moreover, by the first equality in \eqref{4.7} $\ran T=\ran T_1$, which implies that $\ran T$ is closed.

(ii) Assume that $\dim\ker  T<\infty$,  $\ran  T$ is closed and $\dim \wt T/T<\infty$. It is clear that $\dim (\ran \wt T/\ran  T)\leq \dim (\wt T/ T)<\infty$ and, consequently, $\ran\wt T$ is closed. Moreover, $\wh\gN_0(T)=\wh\gN_0(\wt T)\cap T$ and hence $\dim (\wh\gN_0(\wt T)/ \wh\gN_0(T))\leq \dim (\wt T/ T)<\infty$. Since $\dim \wh\gN_0(T)=\dim\ker T<\infty$ and $\dim \wh\gN_0(\wt T)=\dim\ker \wt T$, it follows that $\dim \ker \wt T<\infty$.
\end{proof}
\begin{definition}\label{def4.6}
A symmetric relation  $A\in\C (\gH)$ has a discrete spectrum if $\dim \gN_t(A)<\infty $ and $\ran (A-t)$ is closed for any $t\in\bR$.
\end{definition}
Clearly, $A$ has a discrete spectrum if and only if so is the operator part $A_{\rm op}$ of $A$.

In the following we denote by ${\rm Sym}_d(\gH)$ (${\rm Self}_d(\gH)$) the set of all symmetric (resp. self-adjoint) linear relations $A\in \C(\gH)$ with the discrete spectrum.

If $A\in {\rm Sym}_d(\gH)$, then $n_+(A)=n_-(A)$ and
\begin{gather*}
\bR\setminus \wh \rho(A)=\s_p(A)\in \cF, \qquad \dim \gN_t(A)<\infty, \;\; t\in \s_p(A)
\end{gather*}
In the case $A=A^*\in \C(\gH)$ the inclusion $A\in {\rm Self}_d(\gH)$ is equivalent to conditions
\begin{gather*}
\s(A)=\s_p(A)\in \cF, \qquad \dim \gN_t(A)<\infty, \;\; t\in \s_p(A).
\end{gather*}
\begin{proposition}\label{pr4.7}
Let $A\in \C (\gH)$ be a symmetric relation. Then:

{\rm(i)} if there exists an exit space (in particular,canonical) extension $\wt A\in {\rm Self}_d(\wt\gH)$ of $A$, then $A\in {\rm Sym}_d(\gH)$;

{\rm(ii)} if $A\in {\rm Sym}_d(\gH)$ and  $n_\pm(A)<\infty$, then each symmetric extension $\wt A\in\cex$ belongs to ${\rm Sym}_d(\gH)$.
\end{proposition}
\begin{proof}
If $n_+(A)=n_-(A)<\infty$, then for each $\wt A\in\cex $ and $t\in\bR$ one has $\dim (\wt A-t)/(A-t)=\dim \wt A/A\leq 2 n_+(A)<\infty$.  Now application of Lemma \ref{lem4.5} to relations $A-t$ and $\wt A-t$ yields the result.
\end{proof}
\begin{remark}\label{rem4.8}
For densely defined operators $A$ with finite deficiency indices and canonical extensions $\wt A=\wt A^*$ if $A$ statements of Proposition \ref{pr4.7} are well known (see e.g.\cite[\S14.9]{Nai}).
\end{remark}
In the following theorem we describe in terms of the parameter $\tau\in\RH$ exit space extensions $\wt A_\tau=\wt A_\tau^*$ of $A$ with the discrete spectrum.
\begin{theorem}\label{th4.9}
Let under the assumptions of Theorem \ref{th3.6} $A\in {\rm Sym}_d(\gH)$ and let $\wt A=\wt A_\tau$ be a linear relation in a Hilbert space $\wt\gH\supset\gH$. Then $\wt A\in {\rm Self}_d (\wt\gH)$  if and only if $\tau\in\RM (\cH)$.
\end{theorem}
\begin{proof}
Let $\tau_0\in R [\cH_0]$ be the operator part of $\tau$ (see \eqref{2.11}). Since by Proposition \ref{pr2.7.1} $\dim\cH_0<\infty$, it follows from \cite[Proposition 2.3]{Mog19} that $\tau_0$ admits the  representation
\begin{gather}\label{4.12}
\tau_0(\l)=\begin{pmatrix} \tau_1(\l) & -B_1 \cr -B_1^* & -B_2 \end{pmatrix}:\cH'\oplus\cH''\to \cH'\oplus\cH'', \quad \l\in\CR
\end{gather}
with $\tau_1\in R_u[\cH'], \; B_1\in\B (\cH'',\cH')$ and $B_2=B_2^*\in\B (\cH'')$. Next, assume that  $\gH_r:=\wt\gH\ominus\gH$ and let $\cS\in \cex$ be a symmetric relation in $\gH$ given by $\cS=\wt A\cap\gH^2$. It was shown in the proof of Theorem 3.8 in \cite{Mog19} that there exists a boundary triplet $\Pi'=\{\cH',\G_0',\G_1'\}$ for $\cS^*$ such that $\wt A=\wt\cS_{\tau_1}$ (in the triplet $\Pi'$) and the results of \cite{DM00} imply that there exist a simple symmetric operator $A_r$ in $\gH_r$ and a boundary triplet $\Pi_r=\{\cH',\G_0^r,\G_1^r\}$ for $A_r^*$ such that $\tau_1$ is the Weyl function of $\Pi_r$ and $\wt A\in \ov{\rm ext}\, (S\oplus A_r)$. Hence $\wt A\in\ov{\rm ext}\, (A\oplus A_r),\; n_\pm(A\oplus A_r)<\infty $ and by Proposition \ref{pr4.7} the following equivalences hold:
\begin{gather*}
\wt A\in {\rm Self}_d(\wt\gH)\iff A\oplus A_r\in {\rm Sym}_d(\wt\gH)\iff A_r\in {\rm Sym}_d(\gH_r)\iff A_{0r}\in {\rm Self}_d(\gH_r),
\end{gather*}
where $A_{0r}=\ker\G_0^r$. Moreover, according to \cite[Corollary3.6.2]{BHS} $A_{0r}\in {\rm Self}_d(\gH_r)$ if and only if $\tau_1\in R_{\rm mer}[\cH']$. Finally, by \eqref{4.12} $\tau_1\in R_{\rm mer}[\cH']$ if and only if $\tau\in \RM (\cH)$. These equivalences yield the statement of the theorem.
\end{proof}
In the following theorem we show that in the case of a symmetric operator $A$ in $\gH$ with the discrete spectrum each element $f\in\gH$ admits an eigenvector expansion  due to the eigenvalue problem \eqref{4.3}, \eqref{4.4.1}.
\begin{theorem}\label{th4.10}
Assume that $\dim\gH=\infty$,  $A\in {\rm Sym}_d (\gH)$ is a densely defined operator, $n_+(A)=n_-(A)<\infty$ and $\Pi=\bt$ is a boundary triplet for $A^*$. Moreover, let $C=(C_0,C_1)$ be an entire Nevanlinna pair in $\cH$, let $\tau=\tau_C\in\RM (\cH)$ be given by \eqref{2.20}, let $\nt \; (t\in\bR)$ be the set of all solutions of the problem \eqref{4.3}, \eqref{4.4.1} and let $\wt {EV}$ be the set of all eigenvalues of the same problem.  Then:

{\rm (i)} $\wt {EV}$ is an infinite set without finite limit points, so that it can be written as an increasing infinite sequence $\wt {EV}= \{t_k\}_{\nu_-}^{\nu_+}$. Moreover, $\dim\wt\gN_{t_k}<\infty$ for any $t_k\in  \wt {EV}$.

{\rm (ii)}  for any $f\in\gH$ there exists a sequence $\{f_k\}_{\nu_-}^{\nu_+}$ of eigenvectors $f_k\in \wt\gN_{t_k} \; (t_k\in \wt {EV})$ of the problem \eqref{4.3}, \eqref{4.4.1} such that the following  eigenvector expansion of $f$ is valid:
\begin{gather}\label{4.14}
f=\sum_{k=\nu_-}^{\nu_+}f_k.
\end{gather}
Moreover,  $\wt A_\tau \in {\rm Self}_d(\wt\gH)$ with  $\wt\gH\supset  \gH$  ( for $\wt A_\tau$ see Theorem \ref{th2.10}, (ii)), $\wt {EV} = \s(\wt A_\tau)(=\s_p(\wt A_\tau))$  and the eigenvector  $f_k$ in \eqref{4.14} can be calculated via
\begin{gather*}
f_k=P_\gH E(\{t_k\})f,
\end{gather*}
where $E(\cd)$ is the orthogonal spectral measure of $\wt A_\tau$ and $P_\gH$ is the orthoprojection in $\wt\gH$ onto $\gH$.
\end{theorem}
\begin{proof}
(i) It follows from Theorem \ref{th4.9} that $\wt A_\tau\in {\rm Self}_d(\wt\gH)$. Moreover, by \eqref{4.5} $\wt {EV} = \s(\wt A_\tau)(=\s_p(\wt A_\tau))$. This yields statement (i).

(ii) Let $f\in\gH$.  Then for any $t_k\in \wt {EV}$ one has $E(\{t_k\})f\in \gN_{t_k}(\wt A)$ and by \eqref{4.5} $f_k:=P_\gH E(\{t_k\})f \in \wt\gN_{t_k}$. Since $\wt A_\tau \in {\rm Self}_d(\wt\gH)$, it follows that $f=\sum_{k=\nu_-}^{\nu_+}E(\{t_k\})f$. Therefore
\begin{gather*}
f=P_\gH f=\sum_{k=\nu_-}^{\nu_+}P_\gH E(\{t_k\})f= \sum_{k=\nu_-}^{\nu_+}f_k,
\end{gather*}
which yields \eqref{4.14}.
\end{proof}
\section{Eigenfunction expansions for differential equations}
\subsection{Notations}
Let $\cI=[ a,b\rangle\; (-\infty < a< b\leq\infty)$ be an interval of the real line (the endpoint $b<\infty$  might be either
included  to $\cI$ or not). Denote by $AC(\cI)$ the set of functions $f:\cI\to \bC$ which are absolutely
continuous on each compact interval $[a,b']\subset \cI$.

Assume that $\D:\cI\to \bR$ is a nonnegative function integrable on each compact interval $[a,b']\subset \cI$. Denote  by $\lI$  the semi-Hilbert  space of  Borel measurable functions $f: \cI\to \bC$ satisfying $||f||_\D^2:=\int\limits_{\cI}\D
(x)|f(x)|^2\,dx<\infty$.  The
semi-definite inner product $(\cd,\cd)_\D$ in $\lI$ is defined by $(f,g)_\D=\int\limits_{\cI}\D (x)f(x)\ov{g(x)}\,dx,\quad f,g\in \lI$. Moreover, let $\LI$ be the Hilbert space of the equivalence classes in $\lI$ with respect to the semi-norm $||\cd||_\D$. Denote also by $\pi_\D$  the quotient map from $\lI$ onto
$\LI$. Clearly, $\ker \pi_\D$ coincides with the set of all Borel measurable functions $f:\cI\to \bC$ such that $\D(x) f(x)=0$ (a.e. on $\cI$).

As is known \cite[Ch13.5]{DunSch} each distribution $\xi:\bR\to \B (\bC^r)$ gives rise to the semi-Hilbert space $\lS$ of all Borel-measurable functions $g:\bR\to \bC^r$ such that $||g||_{\lS}^2=\int_\bR (d\xi(t)g(t),g(t))<\infty$. In the following we denote by $\LS$ the Hilbert space  of all equivalence classes in $\lS$ with respect to the seminorm $||\cd||_{\lS}$. Moreover, we denote by $\pi_\xi$ the quotient map from $\lS$ onto $\LS$.

With a $\B (\bC^r)$-valued distribution $\xi$ one associates the multiplication operator   $\L_\xi(=\L_\xi^*)$  in $\LS$.
The orthogonal spectral measure $E_\xi(\cd)$ of $\Lambda_\xi$ is given on Borel sets $\d\subset \bR$ by
\begin {equation}\label{5.0}
E_\xi(\d)\wt g= \pi_\xi (\chi_\d g),\;\;\; \wt g \in\LS ,\;\; g\in \wt g,
\end{equation}
where $\chi_\d$ is the indicator of $\d$.
\subsection{Differential equations with the nontrivial weight}
Assume that $\cI=[a,b\rangle \; (-\infty<a<b\leq \infty)$ is an interval in $\bR$ and let
\begin {equation}\label{5.1}
l[y]= \sum_{k=1}^r  (-1)^k  (p_{r-k}(x)y^{(k)})^{(k)} + p_r(x) y
\end{equation}
be a symmetric differential expression  of  an even order $n=2r$ on $\cI$ with real- valued coefficients $p_j(\cd):\cI\to \bR$. We assume that functions  $p_0^{-1}$ and $p_j, \; j\in\{1,\dots , r\}$ are integrable on each compact interval $[a,b']\subset\cI$ (this means that the endpoint $a$ is regular for $l[y]$).

Following to \cite{Nai,Wei} we denote by $y^{[j]}, \; j\in \{0,1,\; \dots,\; 2r\},$ the quasi-derivatives of a function $y:\cI\to \bC$ (here $y^{[0]}=y$). Denote also by $\dom l$ the set of all functions $y:\cI\to \bC$ such that $y^{[j]}\in\AC $ for  $j\leq2r-1$ and let $l[y]=y^{[2r]},\; y\in\dom l$. With a function $y\in\dom l$ one associates the vector-functions $y^{(j)}:\cI\to \bC^r, \; j\in\{1,2\},$ given by
\begin{gather}\label{5.1.1}
y^{(1)}=y \oplus y^{[1]} \oplus \dots \oplus y^{[r-1]},\; \;\;\;\; y^{(2)}= y^{[2r-1]} \oplus  y^{[2r-2]}\oplus \dots\oplus y^{[r]}.
\end{gather}
We consider the differential equation
\begin{gather}\label{5.2}
l[y]=\l \D(x) y, \quad x\in\cI,\;\;\l\in\bC
\end{gather}
with the weight $\D:\cI\to\bR$   integrable on each compact interval $[a,b']\subset\cI$ and satisfying $\D(x)\geq 0$ a.e.on $\cI$. In the following we assume  that the weight $\D$ is nontrivial  and not necessarily positive (see Definition \ref{def1.0}). A function $y\in\dom l$ is a solution of \eqref{5.2}, if it  satisfies \eqref{5.2}  a.e. on $\cI$. An $m$-component  operator function  \begin{gather}\label{5.3}
Y(x,\l)=(Y_1(x,\l),\, Y_2(x,\l),\, \dots ,\, Y_m(x,\l) ):\bC\oplus\bC \dots \oplus \bC\to \bC,\quad x\in\cI
\end{gather}
with values in $\B(\bC^m,\bC)$ is called an operator solution of \eqref{5.2}, if each component $Y_j(x,\l)$ is a (scalar) solution of \eqref{5.2}. With each such  a solution $Y(x,\l)$ we associate the operator functions $Y^{(j)}:\cI\to \B (\bC^m,\bC^r), \; j\in\{1,2\},$ given by $Y^{(1)}(x,\l)=\bigl( Y_k^{[j-1]}(x,\l) \bigr)_{j=1 \, k=1}^{\;\;r\,\;\;\;m}$ and $Y^{(2)}(x,\l)=\bigl( Y_k^{[2r-j]}(x,\l) \bigr)_{j=1 \, k=1}^{\;\;r\,\;\;\;m}, \; x\in\cI$.

Denote by $\cD_{\rm max}$ the linear manifold in $\lI$ given by
\begin{multline}\label{5.5}
\cD_{\rm max}=\{y\in\dom l\cap \lI: \, l[y]=\D(x) f_y(x) \;\;
(\text{a.e. on}\;\; \cI)\\
 \text{with some} \;\; f_y\in \lI\}.
\end{multline}
Clearly if  $y\in\cD_{\rm max}$ and  $f_{1y}$ and $f_{2y}$ are two functions from \eqref{5.5}, then $\pi_\D f_{1y}=\pi_\D f_{2y}$. Therefore for a given $y\in\cD_{\rm max}$ the function $f_y$ in \eqref{5.5} is defined uniquely up to $\D$-equivalence.

As is known,  for any $y,z\in \cD_{\rm max}$ there exists the limit
\begin{gather*}
[y,z]_b:=\lim_{x\uparrow b} ((y^{(1)}(x), z^{(2)}(x))_{\bC^r}-(y^{(2)}(x), z^{(1)}(x))_{\bC^r}) .
\end{gather*}
This fact enables one to define the linear manifold $\cD_{\rm min}$ in $\lI$ by setting
\begin{gather*}
\cD_{\rm min}=\{y\in \cD_{\rm max}:\, y^{(1)}(a)=y^{(2)}(a)=0\;\; {\rm and}\;\; [y,z]_b=0\;\; \text{for every}\;\; z\in \cD_{\rm max}\}
\end{gather*}

For $\l\in\bC$ denote by $\cN_\l$ the linear space of all solutions $y$ of  \eqref{5.2} belonging to $\lI$ (clearly, $\cN_\l\in \Dma$). It turns out that the number $N_+=\dim \cN_{\l}, \; \l\in\bC_+$  ($N_-=\dim \cN_\l,\; \l\in\bC_-)$ does not depend on $\l\in\bC_+$ (resp. $\l\in\bC_-$). The numbers $N_\pm$ are called the formal deficiency indices of the equation \eqref{5.2}. It turns out that $r\leq N_+=N_-\leq 2r$.  In the following we put $d:=N_\pm$.

Below within this subsection we recall some results from our paper \cite{Mog20} concerning equation \eqref{5.2} with the nontrivial Weight $\D$.
\begin{theorem}\label{th5.2}
For the differential equation \eqref{5.2} the equalities
\begin{gather*}
\wt y=\pi_\D y, \quad \Sma \wt y=\pi_\D f_y, \;\; y\in\Dma;\qquad
\wt y=\pi_\D y, \quad \Smi \wt y=\pi_\D f_y, \;\; y\in\Dmi
\end{gather*}
correctly define the linear operators $\Sma:\dom\Sma\to \LI$ (the maximal operator) and $\Smi:\dom\Smi\to \LI$ (the minimal operator) with the domains $\dom\Sma=\pi_\D\Dma\subset \LI$ and  $\dom\Smi=\pi_\D\Dmi\subset \LI$ respectively. Moreover, $\Smi$ is a closed densely defined symmetric  operator with equal deficiency indices $n_\pm(\Smi)=d<\infty$ and $\Sma=\Smi^*$.
\end{theorem}
\begin{proposition}\label{pr5.4}
Assume that $B=B^*\in\B (\bC^r)$ and let $\cD$ and $\cD_*$ be linear manifolds in $\lI$ given by
\begin{gather}
\cD=\{y\in\Dma: \cos B\cd y^{(1)}(a)+\sin B\cd y^{(2)}(a)=0 \;\; {\rm and}\;\; [y,z]_b=0, \, z\in\Dma   \}\label{5.9}\\
\cD_*=\{y\in\Dma: \cos B\cd y^{(1)}(a)+\sin B\cd y^{(2)}(a)= 0\}\label{5.10}
\end{gather}
Then the equalities
\begin{gather}\label{5.11}
\wt y=\pi_\D y, \quad S \wt y=\pi_\D f_y, \;\; y\in\cD;\qquad
\wt y=\pi_\D y, \quad S^* \wt y=\pi_\D f_y, \;\; y\in\cD_*
\end{gather}
correctly define the linear operators $S:\dom S\to \LI$ and $S^*:\dom S^*\to \LI$ with the domains $\dom S=\pi_\D\cD\subset \LI$ and  $\dom S^*=\pi_\D\cD_*\subset \LI$ respectively. Moreover, $S$ is a closed symmetric  extension of $\Smi$ with the deficiency indices $n_\pm(S)=d-r$ and $S^*$ is the adjoint of $S$.
\end{proposition}
\begin{remark}\label{rem5.4.0}
In the case of the equation \eqref{5.2} with the positive weight $\D$ statements of Theorem \ref{th5.2} and Proposition \ref{pr5.4} are the well-known classical results (see e.g. \cite{Nai,Wei}).
\end{remark}
\begin{proposition}\label{pr5.4.1}
Let $B=B^*\in\B (\bC^r)$, let $S$ and $S^*$ be operators  defined in Proposition \ref{pr5.4} and let $\G_b=(\G_{0b}, \G_{1b})^\top:\Dma \to (\bC^{d-r})^2$ be a surjective linear operator satisfying
\begin{gather}\label{5.12.1}
[y,z]_b=(\G_{0b}y,\G_{1b}z)- (\G_{1b}y,\G_{0b}z),\quad y,z \in \Dma
\end{gather}
(according to \cite{Mog20} such an operator exists). Then:

{\rm(i)} for each $\wt y\in\dom S^*$ there exists a unique $y\in D_*$ such that $\pi_\D y=\wt y$ and $\pi_\D f_y=S^* \wt y$ (here $f_y$ is taken from \eqref{5.5});

{\rm(ii)} the  collection $\Pi=\{\bC^{d-r},\G_0,\G_1\}$ with operators $\G_j: \dom S^*\to \bC^{d-r}$ given by
\begin{gather}\label{5.13}
\G_0\wt y=\G_{0b} y, \qquad \G_1\wt y=-\G_{1b} y, \quad \wt y\in \dom S^*
\end{gather}
is a boundary triplet for $S^*$ (in \eqref{5.13} $y\in \wt y$ is a function from statement {\rm (i)}).
\end{proposition}
In the following with an operator $B=B^*\in\B (\bC^r)$ we associate the $r$-component operator solutions $\f_B(x,\l)=(\f_1(x,\l),\f_2(x,\l), \dots, \f_r(x,\l))(\in \B (\bC^r,\bC))$ and $\psi_B(x,\l)=(\psi_1(x,\l),\psi_2(x,\l), \dots, \psi_r(x,\l))(\in \B (\bC^r,\bC))$ of \eqref{5.2} defined by the initial values
\begin{gather}\label{5.14}
\f_B^{(1)}(a,\l)=\sin B,\;\; \f_B^{(2)}(a,\l)=-\cos B;\;\;\; \psi_B^{(1)}(a,\l)=\cos B,\;\; \psi_B^{(2)}(a,\l)=\sin B.
\end{gather}
It easy to see that for each functions $f\in\lI$ with compact support
the equality
\begin{gather}\label{5.15}
\wh f(t)=\int_\cI \f_B^*(x,t) \D(x)  f(x)\,dx
\end{gather}
defines a continuous functions $\wh f:\bR\to\bC^r$.
\begin{definition}\label{def5.5}
A distribution  $\xi:\bR\to \B(\bC^r)$ is called a spectral function of the equation  \eqref{5.2} if for each function $f\in\lI$ with compact support the  Parseval equality $||\wh f||_{\lS}=||f||_{\D}$ holds.
\end{definition}
If $\xi$ is a spectral function, then for each $f\in\lI$ the integral in \eqref{5.15} converges in $\lS$ to a function $\wh f\in\lS$, which is called the generalized Fourier transform of $f$. Moreover, the equality
\begin{gather}\label{5.18}
V_\xi\wt f =\pi_\xi \wh f, \quad \wt f\in\LI,
\end{gather}
where $\wh f$ is the Fourier transform of a function $f\in\wt f$, defines an isometry $V_\xi\in\B (\LI,$ $\LS)$.

The $m$-component  operator solution $Y(x,\l)$ of \eqref{5.2}  given by  \eqref{5.3} will be referred to the class $\cL_\D^2(\cI;\bC^m)$ if $Y_j(\cd,\l)\in\lI$ for any $j\in\{1,\dots,m\}$. With each such a  solution one associates the operator $Y(\l):\bC^m\to \cN_\l$, given by $(Y(\l)h)(x)=Y(x,\l)h, \; h\in\bC^m$.

A description of all spectral functions of the equation \eqref{5.2} is given by the following theorem.
\begin{theorem}\label{th5.9}$\,$\cite{Mog20}
Let $(\G_{0b}, \G_{1b})^\top$ be the operator from  Proposition \ref{pr5.4.1}. Then:

{\rm (i)} For any $\l\in\CR$ there exists a unique pair of operator solutions $v(\cd,\l)\in\cL_\D^2(\cI;\bC^r)$ and $u(\cd,\l) \in\cL_\D^2(\cI;\bC^{d-r})$ of \eqref{5.2} satisfying the boundary conditions:
\begin{gather*}
\cos B\cd v^{(1)}(a,\l)+ \sin B\cd v^{(2)}(a,\l)=-I_r,\qquad \G_{0b}v(\l)=0,\quad \l\in\CR\\
\cos B\cd u^{(1)}(a,\l)+ \sin B\cd u^{(2)}(a,\l)=0,\qquad \G_{0b}u(\l)=I_{d-r},\quad \l\in\CR
\end{gather*}
Moreover, the equalities
\begin{gather}
M(\l)=\begin{pmatrix} m_0(\l) & M_2(\l)\cr M_3(\l) & M_4(\l)  \end{pmatrix}: \bC^r\oplus\bC^{d-r}\to \bC^r\oplus\bC^{d-r},\quad \l\in\CR\label{5.22}\\
m_0(\l)=\sin B\cd v^{(1)}(a,\l)- \cos B\cd v^{(2)}(a,\l)\nonumber\\
 M_2(\l)=\sin B\cd u^{(1)}(a,\l)- \cos B\cd u^{(2)}(a,\l),\;\;\;
M_3(\l)=-\G_{1b} v(\l), \;\;\; M_4(\l) =-\G_{1b} u(\l)\nonumber
\end{gather}
define a function $M\in R[\bC^r\oplus\bC^{d-r}]$ with $M_4\in R_u[\bC^{d-r}]$.

{\rm (ii)} For any $\tau\in\wt R(\bC^{d-r})$ the equality
\begin{gather}\label{5.23}
m(\l)=m_0(\l)-M_2(\l)(\tau(\l)+M_4(\l))^{-1}M_3(\l),
\quad\l\in\CR
\end{gather}
defines a function $m\in R [\bC^r]$ such that the spectral function  $\xi=\xi_\tau$ of $m$ is a spectral function of the equation \eqref{5.2} and, conversely, for any spectral function $\xi $ of \eqref{5.2} there is a unique $\tau\in\wt R(\bC^{d-r})$ such that $\xi$ is a spectral function of $m$ given by \eqref{5.23}.
\end{theorem}
\begin{theorem}\label{th5.10}$\,$\cite{Mog20}
Let under the assumptions of Proposition \ref{pr5.4.1} $\Pi=\{\bC^{d-r},\G_0,\G_1\}$ be the boundary triplet \eqref{5.13} for $S^*$. Assume also that $\tau\in\wt R(\bC^{d-r})$, $\wt S_\tau=\wt S_\tau^*\in \cC (\wt\gH ) \; (\wt\gH\supset \LI)$ is the corresponding exit space extension of $S$, $\xi=\xi_\tau$ is the spectral function of the equation \eqref{5.2} (see Theorem \ref{th5.9}, {\rm (ii)}) and $\L_\xi$ is the multiplication operator in $\LS$. Then there exists a unitary operator $\wt V\in \B (\wt\gH, \LS)$ such that $\wt V\up \LI=V_\xi$ and the operators $\wt S_\tau $  and $\L_\xi$ are unitarily equivalent by means of $\wt V$.
\end{theorem}
\subsection{Eigenfunction expansions: calculation of eigenfunctions and uniform convergence}\label{sub5.3}
Below we  suppose that for equation \eqref{5.2} the following assumptions hold:

(A1) $\Dma\subset\lI$ is the linear manifold \eqref{5.5} and $(\G_{0b},\G_{1b})^\top:\Dma\to (\bC^{d-r})^2$ is a surjective operator satisfying \eqref{5.12.1}.

(A2) $B=B^*\in\B (\bC^r)$ and $C=(C_0,C_1)$ is an entire Nevanlinna pair  in $\bC^{d-r}$.

Consider the eigenvalue problem
\begin{gather}
l[y]=t\D(x)y\label{5.30}\\
(\cos B) y^{(1)}(a)+(\sin B) y^{(2)}(a)=0,\qquad  C_0(t)\G_{0b}y+C_1(t)\G_{1b}y=0\label{5.31}
\end{gather}
with separated boundary conditions \eqref{5.31} depending on the parameter $t\in\bR$.  The set of all solutions of the problem \eqref{5.30}, \eqref{5.31} for a given $t\in\bR$ (i.e., the set of all $y\in\cN_t$ satisfying \eqref{5.31}) will be denoted by $\gN_t$. Clearly, $\gN_t$ is a finite dimensional subspace in  $\lI$.
\begin{definition}\label{def5.11}
A point $t\in\bR$ such that $\gN_t\neq \{0\}$ is called an eigenvalue of the problem \eqref{5.30}, \eqref{5.31}.  The subspace $\gN_t (\subset \lI)$ for an eigenvalue  $t$ is called an eigenspace and a function $y\in\gN_t$ is called an eigenfunction of the problem \eqref{5.30}, \eqref{5.31}.
\end{definition}
\begin{theorem}\label{th5.12}
Assume that for equation \eqref{5.2} with the nontrivial weight $\D$ the operator $\Smi$ has the discrete spectrum. Moreover let the assumptions {\rm (A1)} and {\rm (A2)} at the beginning of the subsection hold, let $\gN_t\; (t\in\bR)$ be the set of all  solutions of the eigenvalue problem \eqref{5.30}, \eqref{5.31} and let $EV$ be the set of all eigenvalues of the same problem. Then:

{\rm (i)} $  EV$ is an infinite countable subset of $\bR$ without finite limit  points, so that it can be written as an increasing infinite sequence $  EV=\{t_k\}_{\nu_-}^{\nu_+}, \; t_k\in\bR$;

{\rm (ii)}  for any $y\in\lI$ there exists a sequence $\{y_k\}_{\nu_-}^{\nu_+}$ of eigenfunctions $y_k\in\gN_{t_k} \; (t_k\in EV)$ of the problem \eqref{5.30}, \eqref{5.31} such that the following eigenfunction expansion holds:
\begin{gather}\label{5.33}
y(x)=\sum_{k=\nu_-}^{\nu_+}y_k(x).
\end{gather}
The series in \eqref{5.33} converges in $\lI$, that is
\begin{gather}\label{5.34}
\lim_{{\nu_1\to \nu_-}\atop{\nu_2\to \nu_+}} \left|\left|y- \sum_{k=\nu_1}^{\nu_2}y_k \right|\right|_\D=0.
\end{gather}

{\rm (iii)} Let $\tau=\tau_C\in\RM (\bC^{d-r})$ be given by \eqref{2.20}. Then  the function $y_k$ in \eqref{5.33} can be defined as a unique function in $\gN_{t_k}$ such that
\begin{gather}\label{5.35}
\pi_\D y_k=P E(\{t_k\})\pi_\D y.
\end{gather}
Here $E(\cd)$ is the orthogonal spectral measure of the exit space extension $\wt S_\tau=\wt S_\tau^*\in \cC (\wt\gH)$ of $S$ $\;(\wt\gH\supset\LI)$ corresponding to $\tau$ in the boundary triplet $\Pi=\{\bC^{d-r}, \G_0,\G_1\}$ for $S^*$  (see Propositions \ref{pr5.4} and \ref{pr5.4.1}) and $P$  is the orthoprojection in $\wt\gH$ onto $\LI$. Moreover, $\wt S_\tau\in {\rm Self}_d (\wt\gH)$ and $  EV=\s(\wt S_\tau)(=\s_p(\wt S_\tau))$.
\end{theorem}
\begin{proof}
Together with \eqref{5.30}, \eqref{5.31} consider the abstract eigenvalue problem
\begin{gather}\label{5.37}
S^* \wt y=t\wt y, \qquad C_0(t)\G_0 \wt y-C_1(t)\G_{1}\wt y=0.
\end{gather}
Denote by $\wt\gN_t \; (t\in\bR)$ the set of all solutions of the problem \eqref{5.37} and by $\wt {EV}$ the set of all eigenvalues of the same problem. In view of Theorem \ref{th4.10} the following assertions  are valid:

(a1) $\wt S_\tau\in {\rm Self}_d (\wt\gH)$ and $\wt {EV}=\s(\wt S_\tau)(=\s_p(\wt S_\tau))$;

(a2) $\wt {EV}=\{t_k\}_{\nu_-}^{\nu_+}$ is an  infinite set without finite limit  points and for any $\wt y\in\LI$ there exists a sequence $\{\wt y_k\}_{\nu_-}^{\nu_+}, \; \wt y_k\in \wt\gN_{t_k}\; (t_k\in \wt {EV}) $ such that
\begin{gather}\label{5.38}
\wt y=\sum_{k=\nu_-}^{\nu_+} \wt y_k
\end{gather}
(the series converges in $\LI$). Moreover, $\wt y_k$ in \eqref{5.38} can be defined via
\begin{gather}\label{5.39}
\wt y_k=P E(\{t_k\})\wt y.
\end{gather}
By using \eqref{5.11} and \eqref{5.13} one can easily verify that $\pi_\D \gN_t=\wt \gN_t,\; t\in\bR$. Moreover, according to \cite[Proposition 5.11]{Mog20} the equation \eqref{5.2} is definite, that is the equalities $l[y]=t \D(x) y(x)$ and $\D(x)y(x)=0$ (a.e. on $\cI$) yields $y=0$. Hence $\ker (\pi_\D\up \gN_t)=\{0\}$, so that $\pi_\D\up \gN_t \; (t\in\bR)$ is an isomorphism of $\gN_t$ onto $\nt$.  Therefore $  EV=\wt {EV}$, which in view of (a1) and (a2) yields statement (i) and the equality $  EV=\s(\wt S_\tau)$.

Next assume that  $y\in\lI$. Then by assertion (a2) $\wt y:=\pi_\D y $ admits the representation \eqref{5.38} with $\wt y_k\in\wt\gN_{t_k}$ given by \eqref{5.39}. As was shown, there exists a unique  $y_k\in\gN_{t_k}$ such that $\pi_\D y_k=\wt y_k$, that is \eqref{5.35} holds. Moreover, \eqref{5.38} yields \eqref{5.34}. This proves statements (ii) and (iii).
\end{proof}
Our next goal is to provide an explicit method for calculation of eigenfunctions $y_k$ in the expansion  \eqref{5.33}. To this end we need the following definition.
\begin{definition}\label{def5.14}
A discrete distribution $\xi:\bR\to \B(\bC^r)$ which is a spectral function of the equation \eqref{5.2} in the sense of Definition \ref{def5.5} is called a discrete spectral function of this equation.
\end{definition}
If $\xi=\{F,\Xi\}$ is a discrete spectral function of \eqref{5.2} with $F=\{t_k\}_{\nu_-}^{\nu_+} \in\cF$ and $\Xi=\{\xi_k\}_{\nu_-}^{\nu_+},\; 0\leq\xi_k\in\B (\bC^r),\; \xi_k\neq 0$,  then  a function $g\in\lS$ can be identified with a sequence $g=\{g_k\}_{\nu_-}^{\nu_+} \; (g_k=g(t_k)\in\bC^r)$ such that  $\sum\limits_{k=\nu_-}^{\nu_+}(\xi_k g_k,g_k)\leq\infty$.

Assume  that $\xi=\{F,\Xi\}$ is a discrete spectral function of \eqref{5.2} with $F=\{t_k\}_{\nu_-}^{\nu_+}$ and $\Xi=\{\xi_k\}_{\nu_-}^{\nu_+}$. Then  the generalized Fourier transform $\wh f$ of a function $f\in\lI$ with compact support (see \eqref{5.15})  can be identified with a sequence $\wh f= \{\wh f_k\}_{\nu_-}^{\nu_+}$, $\wh f_k=\wh f(t_k)\in\bC^r$, given by
\begin{gather}\label{5.41}
\wh f_k=\int_\cI\f_B^*(x,t_k)\D(x)f(x)\, dx, \quad t_k\in F.
\end{gather}
(here the integral exists as the Lebesgue integral). The sequence $\wh f=\{\wh f_k\}_{\nu_-}^{\nu_+} $ satisfies the Parseval equality
\begin{gather*}
(||f||^2_\D=)\int_\cI\D(x)|f(x)|^2\, dx=\sum_{k=\nu_-}^{\nu_+}(\xi_k\wh f_k,\wh f_k)\,\left(=||\wh f||_{\lS}^2\right)
\end{gather*}
For an arbitrary function $f\in\lI$ its Fourier transform is a sequence $\wh f= \{\wh f_k\}_{\nu_-}^{\nu_+}\in\lS$ with $\wh f_k=\wh f(t_k)\in\bC^r$, such that
\begin{gather*}
\lim_{b'\to b} \sum_{k=\nu_-}^{\nu_+}(\xi_k(\wh f_k - \smallint_{[a,b']} \f_B^*(x,t_k)\D(x)f(x)\,dx),\wh f_k- \smallint_{[a,b']} \f_B^*(x,t_k)\D(x)\,f(x))=0
\end{gather*}
The last equality will be written in the form  of the integral \eqref{5.41} as well (one says that this integral converges in $\lS$). The vectors $\wh f_k\in\bC^r$ (see \eqref{5.41}) will be called the Fourier coefficients of a function $f\in\lI$ (with respect to the spectral function $\xi$).

For a discrete spectral function $\xi=\{F,\Xi\}$ with $F=\{t_k\}_{\nu_-}^{\nu_+}$ and $\Xi=\{\xi_k\}_{\nu_-}^{\nu_+}$ the isometry $V_\xi\in \B (\LI,\LS)$ is defined by the same formula \eqref{5.18}, but now $\wh f=\{\wh f_k\}_{\nu_-}^{\nu_+}$ is a sequence \eqref{5.41}. Moreover, \cite[Proposition 3.8]{Mog20} implies that
\begin{gather}\label{5.42}
V_\xi^*\wt g=\pi_\D \left(\sum_{k=\nu_-}^{\nu_+} \f_B(\cd,t_k)\xi_k g_k\right), \quad \wt g\in\LS,\;\; g=\{g_k\}_{\nu_-}^{\nu_+}\in\wt g,
\end{gather}
where the series converges in $\lI$.

In the following theorem we specify explicit formulas for calculations of  eigenfunctions  $y_k$ in the  expansion \eqref{5.33} of $y$.
\begin{theorem}\label{th5.17}
Let the assumptions be the same as in Theorem \ref{th5.12}. Moreover, let $M\in R[\bC^r\oplus\bC^{d-r}]$ be the operator function \eqref{5.22}. Then:

{\rm (i)} The equality
\begin{gather}\label{5.43}
m(\l)=m_0(\l)+M_2(\l)(C_0(\l)- C_1(\l)M_4(\l))^{-1}C_1 (\l)M_3(\l),
\quad\l\in\CR
\end{gather}
defines a function $m\in \Rm [\bC^r]$ and according to Assertion \ref{ass2.6.2} the (discrete) spectral function $\xi$ of $m$ is $\xi=\{F_{m},\, \Xi\}$, where $F_{m}=\{t_k\}_{\nu_-}^{\nu_+}$ is the set of all poles of $m$ and $\Xi=\{\xi_k\}_{\nu_-}^{\nu_+}$ with $\xi_k\in \B(\bC^r)$ given by $\xi_k=-\underset{t_k}{\rm res} \,m$. Moreover, $\xi $ is a discrete spectral function of the equation \eqref{5.2}.

{\rm (ii)} $  EV=F_{m}=\{t_k\}_{\nu_-}^{\nu_+}$ and an eigenfunction $y_k$ in the expansion \eqref{5.33} of $y\in\lI$ admits the representation
\begin{gather}\label{5.44}
y_k(x)=\f_B(x,t_k) \xi_k \wh y_k,
\end{gather}
where $\wh y_k$ is the Fourier coefficient \eqref{5.41} of $y$ (with respect to the spectral function $\xi=\xi$).
\end{theorem}
\begin{proof}
Let $\tau=\tau_C\in \RM (\bC^{d-r})$ be given by \eqref{2.20}. One can easily verify that
\begin{gather*}
-(\tau(\l)+M_4(\l))^{-1}=(C_0(\l)- C_1(\l)M_4(\l))^{-1}C_1(\l), \quad \l\in \CR
\end{gather*}
and therefore \eqref{5.23} can be written in the form \eqref{5.43}. According to Theorem \ref{th5.9} $m\in R[\bC^r]$ and the spectral function $\xi$ of $m$ is a spectral function of the equation \eqref{5.2}.

Assume that $\Pi=\{\bC^{d-r}, \G_0,\G_1\}$ is the boundary triplet  \eqref{5.13} for $S^*$ (see Propositions \ref{pr5.4} and \ref{pr5.4.1}) and let $\wt S_\tau=\wt S_\tau^*\in\cC (\wt\gH) \; (\wt\gH\supset \LI)$ be an exit space extension of  $S$ corresponding to $\tau$. Then according to Theorem \ref{th5.12} $\wt S_\tau\in {\rm Self}_d(\wt\gH)$ and $ EV=\s(\wt S_\tau)$. Moreover, by Theorem  \ref{th5.10} $\wt S_\tau$ is unitarily equivalent to the multiplication operator $\L_\xi$ in $\LS$ and hence $\L_\xi$ has the discrete spectrum. Therefore $\xi=\{F,\Xi\}$ is a discrete spectral function with $F=\s(\L_\xi)\in\cF$ and Assertion  \ref{ass2.6.2} implies that $m\in\Rm [\bC^r]$ and  $F=F_{m}$. On the other hand, $\s(\L_\xi)=\s(\wt S_\tau)=  EV$ and, consequently, $ EV=F_{m}$.

Now it remains to prove \eqref{5.44}. Let $E(\cd)$ be the orthogonal spectral measure of $\wt S_\tau$ and let $P$ be the orthoprojection in $\wt\gH$ onto $\LI$. Then by Theorem \ref{th5.10} for any $t_k\in   EV$ one has
\begin{gather}\label{5.45}
P E(\{t_k\})\up \LI=V_\xi^* E_\xi(\{t_k\})V_\xi,
\end{gather}
where $ V_\xi\in\B (\LI,\LS)$ is the isometry \eqref{5.18} and $E_\xi(\cd)$ is the orthogonal spectral measure of $\L_\xi$ given by \eqref{5.0}. According to Theorem \ref{th5.12} $y_k$ is defined by \eqref{5.35}. Combining \eqref{5.45} with \eqref{5.35}  one gets
\begin{gather}\label{5.46}
\pi_\D y_k = V_\xi^* E_\xi(\{t_k\}) V_\xi \pi_\D y, \quad y\in\lI.
\end{gather}
It follows from \eqref{5.18} that $ V_\xi \pi_\D y =\pi_\xi \wh y$, where $\wh y=\{\wh y_k\}_{\nu_-}^{\nu_+} \; (\wh y_k=\wh y(t_k))$ is the sequence of the Fourier coefficients of $y$. Substituting this equality into \eqref{5.46} and taking \eqref{5.0} and \eqref{5.42} into account one obtains $\pi_\D y_k=\pi_\D (\f_B(\cd,t_k)\xi_k \wh y_k)$. This yields \eqref{5.44}.
\end{proof}
Let $\dim\cH<\infty$ and let $C=(C_0,C_1)$ be a Nevanlinna pair in $\cH$. Then according to Lemma \ref{lem2.6.7} the subspace $\cK:=\ker C_1(\l)\subset\cH$ does not depend on $\l\in\CR$. Let $\cH_0:=\cH\ominus\cK$, let $C_{0j}(\l), \; j\in\{0,1\},$ and $C_{10}(\l)$ be entries of the block representations  \eqref{2.21} of $C_0(\l)$ and $C_1(\l)$ and let $\wh C_1(\l)\in \B (\cH)$ be given by \eqref{2.22}. Then according to Lemma \ref{lem2.6.7} $\ker \wh C_1(\l)=\{0\}$ and the relations \eqref{2.23} define a Nevanlinna function $\tau_0\in R[\cH_0]$. Let $\cB_{\tau_0 \infty}\in\B (\cH_0)$ and $D_{\tau_0 \infty}:\dom D_{\tau_0 \infty}\to \cH_0\; (\dom D_{\tau_0 \infty}\subset\cH_0)$ be operators corresponding to $\tau_0$ in accordance with Proposition \ref{pr2.4}, {\rm (i)}. In the following with a pair $C$ we associate a linear relation $\eta_C$ in $\cH$ given by
\begin{gather}\label{5.50}
\eta_C=\{\{h, - D_{\tau_0 \infty}h + \cB_{\tau_0 \infty}h_0+k\}:h\in\dom D_{\tau_0 \infty}, h_0\in\cH_0, k\in\cK \}.
\end{gather}
If $\ker C_1(\l)=\{0\}$ for some (and hence all) $\l\in\CR$, then definition of $\eta_C$ can be rather simplified. Namely, in this case $\wh C_1(\l)=C_1(\l), \; C_{00}(\l)=C_0(\l)$ and hence $\tau_0(\l)=\tau_C(\l)=-C_1^{-1}(\l)C_0(\l), \; \l\in\CR,$ and
\begin{gather}\label{5.51}
\eta_C=\{\{h, - D_{\tau_0 \infty}h + \cB_{\tau_0 \infty}h_0\}:h\in\dom D_{\tau_0 \infty}, h_0\in\cH_0\}.
\end{gather}
In the following theorem we provide sufficient conditions for the uniform convergence of the eigenfunction expansion \eqref{5.33}.
\begin{theorem}\label{th5.19}
Let under the assumptions of Theorem \ref{th5.12} $\eta_C$ be the linear relation in $\bC^{d-r}$ defined by \eqref{5.50} and let  $y\in\dom l\cap\lI$ be  a function  such  that: {\rm (i)} the  equality $l[y]=\D(x) f_y(x)$ (a.e. on  $\cI$) holds with some $f_y\in\lI$; {\rm (i)} the boundary conditions
\begin{gather}
(\cos B) y^{(1)}(a)+(\sin B) y^{(2)}(a)=0, \qquad \{\G_{0b}y, - \G_{1b} y\}\in \eta_C\label{5.54}
\end{gather}
are satisfied. Then
\begin{gather}\label{5.55}
y^{[j]}(x)=\sum_{k=\nu_-}^{\nu_+} y_k^{[j]}(x), \quad j\in \{0,1, \dots, 2r-1\},
\end{gather}
where $y_k$ are eigenfunctions from the expansion \eqref{5.33} of $y$.  The series in \eqref{5.55} converges absolutely for any $x\in\cI$ and uniformly on each compact interval $[a,b']\subset \cI$.
\end{theorem}
\begin{proof}
Let $\tau=\tau_C\in\RM(\bC^{d-r})$ be given by \eqref{2.20} and let $\tau_0$ and $\cK$ be the operator and multivalued parts of $\tau$ respectively. Then by Lemma \ref{lem2.6.7} $\cK=\ker C_1(\l)$ and $\tau_0(\l)=-\wh C_1^{-1}(\l)C_{00}(\l), \; \l\in\CR$. This implies that $\eta_C=\eta_\tau$, where $\eta_\tau\in\C(\bC^{d-r})$   is the linear relation defined in \cite[Theorem 2.4]{Mog20}.

Assume that $\xi$ is a discrete spectral function of \eqref{5.2} defined in Theorem \ref{th5.17}. As was shown in the proof of this theorem equality \eqref{5.43} is equivalent to \eqref{5.23} and according to Theorem \ref{th5.9} $\xi$ is a spectral function corresponding to $\tau$ in the sense of \cite{Mog20} (see Definition \ref{def5.5}). Moreover, in view of  \eqref{5.44}  the expansion \eqref{5.33} can be written as the inverse Fourier transform
\begin{gather*}
y(x)=\int_\bR \f_B (x,t)\, d \xi(t) \wh y(t),
\end{gather*}
where $\wh y$ is the Fourier transform \eqref{5.15}  of  $y$.  Now the required statement follows from  \cite[Theorem 5.14]{Mog20}.
\end{proof}
In the case of a selfadjoint eigenvalue problem \eqref{1.1}, \eqref{1.3} the above results take a rather simpler form. Namely, the following corollary holds.
\begin{corollary}\label{cor5.19.1}
Assume that for equation \eqref{5.2} with the nontrivial weight $\D$ the operator $\Smi$ has the discrete spectrum and let the assumption {\rm (A1)} at the beginning of Section \ref{sub5.3} be satisfied. Moreover, let $\gN_t\; (t\in\bR)$ be the set of all solutions of the selfadjoint eigenvalue  problem \eqref{1.1}, \eqref{1.3} with $\l=t\in\bR,\; B=B^*\in\B(\bC^r), \; B_1=B_1^*\in\B(\bC^{d-r})$ and let $ EV$ be the set of all eigenvalues of this problem. Then:

{\rm(i)} Statements of Theorem \ref{th5.12} are valid and, moreover, eigenspaces $\gN_{t_k}$ are mutually orthogonal in $\lI$.

{\rm(ii)} Statements of Theorem \ref{th5.17} hold with $C_0(\l)=\cos B_1$ and $C_1(\l)=\sin B_1$ in  \eqref{5.43}.

{\rm(iii)} Assume that $y\in\dom l\cap\lI$ is a function such that the equality $l[y]=\D(x) f_y(x)$ holds a.e. on $\cI$ with some $f_y\in\lI$ and the boundary conditions \eqref{1.3} are satisfied. Then statements of Theorem \ref{th5.19} are valid.
\end{corollary}
\begin{proof}
Clearly, the equalities $C_0(\l)=\cos B_1, \; C_1(\l)=\sin B_1,\; \l\in\bC,$ defines a pair $C=(C_0,C_1)\in {\rm ENP}\, (\bC^{d-r}) $, for which the boundary conditions \eqref{5.31} take the form \eqref{1.3}. Therefore  statements (i) and (ii) are  implied by Theorems \ref{th5.12} and \ref{th5.17} respectively. Next, $\tau_C(\l)=\t,\; \l\in\bC,$ with $\t=\t^*=\{\{h,h'\}\in \cH^2: (\cos B_1) h+ (\sin B_1) h'=0\}$ and hence $\cK=\ker (\sin B_1)=\mul \t, \; \tau_0(\l)=\t_{\rm op}, \; \l\in\CR$. Clearly, $\cB_{\tau_0\infty}=0$ and the equality $\im \tau_0(\l)=0$ yields $\dom D_{\tau_0\infty}=\cH_0$ and $D_{\tau_0\infty}=\t_{\rm op}$. Therefore by \eqref{5.50} $\eta_C=-\t $ and the second boundary condition in \eqref{5.54} takes the form $\{\G_{0b}y,\G_{1b}y\}\in\t$, which is equivalent to the second equality in \eqref{1.3}. Now statement (iii) follows from Theorem \ref{th5.19}.
\end{proof}
\subsection{The case of a quasiregular  equation}
Recall that  differential equation \eqref{5.2} is called  quasiregular if it has the maximal formal deficiency indices $d=2r$ and regular if it is given on a compact interval $\cI=[a,b]$ (the latter implies that $p_0^{-1},\; p_j, \; j\in\{1,2, \dots, r\}$ and $\D$ are integrable on $\cI$). Clearly, each regular equation \eqref{5.2} is quasiregular.
\begin{proposition}\label{pr5.20}
If equation \eqref{5.2} with the nontrivial weight $\D$ is quasiregular, then the operator $\Smi$ has the discrete spectrum.
\end{proposition}
\begin{proof}
Assume that \eqref{5.2} is quasiregular. Then according to \cite[Proposition 5.2]{Mog20} there exists a quasiregular  definite Hamiltonian system on  $\cI$ such that the minimal operator $\Tmi$ generated by this system is unitarily equivalent to $\Smi$. Since $\Tmi$ has the discrete spectrum (see e.g. \cite{BHS,LesMal03}), so is $\Smi$. \end{proof}
\begin{remark}\label{rem5.21}
For the quasiregular equation \eqref{5.2} $\f_C(\cd,t_k)\in\cL_\D^2(\cI;\bC^r)$ and hence for any $f\in\lI$ the integral in \eqref{5.41} exists as the Lebesgue integral. Therefore in this case the Fourier coefficients $\wh f_k$ of $f$ are defined by \eqref{5.41} independently on the spectral function $\xi$.
\end{remark}
According to \cite{Hol85} for quasiregular equation \eqref{5.2} the  equalities
\begin{gather}
\G_{0b}y:=\lim_{x\to b} ( (\psi_C^{(2)}(x,0))^*y^{(1)}(x) - (\psi_C^{(1)}(x,0))^*y^{(2)}(x))  \label{5.59}\\
\G_{1b}y:=\lim_{x\to b} [ -(\f_C^{(2)}(x,0))^*y^{(1)}(x) + (\f_C^{(1)}(x,0))^*y^{(2)}(x)] , \quad y\in\Dma. \label{5.60}
\end{gather}
define a surjective operator $(\G_{0b}, \G_{1b})^\top:\Dma\to (\bC^r)^2$ satisfying \eqref{5.12.1}.

In the following two propositions we provide some peculiarities of the previous results for quasiregular equations.
\begin{proposition}\label{pr5.22}
Assume  that  equation \eqref{5.2} with the nontrivial weight $\D$ is quasiregular and let assumptions {\rm (A1)} and {\rm (A2)} at the beginning of Section \ref{sub5.3} be satisfied with the operator $(\G_{0b}, \G_{1b})^\top$ given by \eqref{5.59} and \eqref{5.60}. Moreover, let $\gN_t (t\in\bR)$ be the set of all solutions of the eigenvalue problem \eqref{5.30}, \eqref{5.31} and let $EV$ be the set of all eigenvalues of the same problem.  Then statements of  Theorems \ref{th5.12}, \ref{th5.17} and \ref{th5.19} are valid. Moreover, in this case  the operator-function $m\in\Rm [\bC^r]$  in Theorem  \ref{th5.17} can be calculated via the linear-fractional transform
\begin{gather} \label{5.61}
m(\l)=(C_0(\l)w_1(\l)+C_1(\l)w_3(\l))^{-1}
(C_0(\l)w_2(\l)+C_1(\l)w_4(\l)), \quad \l\in\CR
\end{gather}
with the operator coefficients $w_j(\l)(\in \B (\bC^r))$ given by
\begin{gather*}
w_1(\l)=I_r+\l\smallint_\cI \psi_C^*(x,0)\D(x) \f_C(x,\l)\,dx, \qquad w_2(\l)=\l\smallint_\cI \psi_C^*(x,0)\D(x) \psi_C(x,\l)\,dx\\
w_3(\l)=-\l\smallint_\cI \f_C^*(x,0)\D(x) \f_C(x,\l)\,dx, \qquad w_4(\l)=I_r-\l\smallint_\cI \f_C^*(x,0)\D(x) \psi_C(x,\l)\,dx.
\end{gather*}
\end{proposition}
\begin{proof}
Application of \cite[Theorem 6.16
and Proposition 6.11]{Mog15} to the same Hamiltonian system as in the proof of Proposition \ref{pr5.20} gives  \eqref{5.61}. This  and  Proposition \ref{pr5.20} yield the result.
\end{proof}
For regular equation \eqref{5.2} one can put in the assumption (A1)
$\G_{0b}y=y^{(1)}(b)$ and $\G_{1b}y= y^{(2)}(b),\;y\in\Dma$. In this case the boundary conditions \eqref{5.31} take the form
\begin{gather}
(\cos B) y^{(1)}(a)+(\sin B) y^{(2)}(a)=0,\qquad  C_0(t)y^{(1)}(b)+C_1(t)y^{(2)}(b)=0.\label{5.63}
\end{gather}
\begin{proposition}\label{pr5.23}
Let  equation \eqref{5.2} with the nontrivial weight $\D$ be regular and let the assumption {\rm (A2)} be satisfied.  Then statements of Theorems \ref{th5.12}, \ref{th5.17} and \ref{th5.19} are valid for the eigenvalue problem \eqref{5.30}, \eqref{5.63}. Moreover, in this case the boundary conditions \eqref{5.54} take the form
\begin{gather*}
(\cos B) y^{(1)}(a)+(\sin B) y^{(2)}(a)=0, \qquad \{y^{(1)}(b), -y^{(2)}(b)\}\in\eta_C
\end{gather*}
and the operator function $m\in\Rm [\bC^r]$ in Theorem  \ref{th5.17} can be calculated via
\begin{gather*}
m(\l)=(C_0(\l)\f_B^{(1)}(b,\l)+C_1(\l)\f_B^{(2)}(b,\l))^{-1}(C_0(\l)
\psi_B^{(1)}(b,\l)+C_1(\l)\psi_B^{(2)}(b,\l)),  \;\; \l\in\CR.
\end{gather*}
\end{proposition}
\begin{proof}
The required statements are implied by  Proposition \ref{pr5.20} and \cite[Theorem 5.15]{Mog20}.
\end{proof}
In the case $r=1$  \eqref{5.2} takes the form of the Sturm-Liouville equation \eqref{1.6}. Below we prove  Theorems \ref{th1.2}, \ref{th1.3} and Remark \ref{rem1.4} concerning quasiregular equation \eqref{1.6}.
\begin{proof}
In view of \eqref{5.51} the  linear relation $\eta_C$ in $\bC$ is defined as follows: (i) in Case 1 $\eta_C=\{0\}\oplus\bC$; (ii) in Case 2 $\eta_C=\{\{h,-D_\infty h\}:h\in\bC\}$; (iii) in Case 3 $\eta_C=\{0\}$. Now application of Proposition \ref{pr5.22} to the eigenvalue problem \eqref{1.6}, \eqref{1.11} gives Theorems \ref{th1.2} and \ref{th1.3}. Finally, Remark \ref{rem1.4} is implied by Proposition \ref{pr5.23}.
\end{proof}
\subsection{Example}
Consider the eigenvalue problem
\begin{gather}
-y''=\l y, \quad x\in\cI=[0,1], \;\;\l\in\bC\label{5.67}\\
y'(0)=0, \qquad \l y(1)-y'(1)=0\label{5.68},
\end{gather}
i.e., the eigenvalue problem \eqref{1.6}, \eqref{1.11} with $p(x)\equiv 1, \; q(x)\equiv 0, \; \D(x)\equiv 1, \; x\in\cI=[0,1],$ in \eqref{1.6} and $B=\frac \pi 2, \; C_0(\l)=\l, \; C_1(\l)=-1$ in \eqref{1.11}. Since the equation \eqref{5.67} is regular and $C=(C_0,C_1)\in {\rm ENP}(\bC)$, we may  apply Theorems \ref{th1.2}, \ref{th1.3} and Remark \ref{rem1.4} to the problem \eqref{5.67},\eqref{5.68}.

The immediate checking shows that the solutions $\f_B(\cd,\l)$ and $\psi_B(\cd,\l)$ of \eqref{5.67} are $\f_B(x,\l)=\cos (\sqrt\l\, x)$ and $\psi_B(x,\l)=\tfrac 1 {\sqrt \l} \sin (\sqrt\l\, x)$. Hence
\vskip 0.5mm
\centerline{$\f_B(1,\l)=\cos \sqrt\l, \qquad  \f_B'(1,\l)=-\sqrt \l \sin \sqrt\l$}
\vskip 0.5mm
\centerline{$\psi_B(1,\l)=\tfrac 1 {\sqrt \l} \sin \sqrt\l, \qquad \psi_B'(1,\l)=\cos \sqrt\l $}
\vskip 0.5mm
\noindent and according to Theorem \ref{th1.2} and Remark \ref{rem1.4} the equality $m(\l)=\displaystyle {{\Phi(\l)}\over{\Psi (\l)}},\; \l\in\CR,$ with entire functions
\vskip 0.5mm
\centerline{$\Phi(\l)=\psi_B(1,\l)\, C_0(\l)+\psi_B'(1,\l)\, C_1(\l)
=\sqrt\l \sin \sqrt\l- \cos \sqrt\l$}
\vskip 1mm
\centerline{$\Psi(\l)=\f_B(1,\l)\,C_0(\l)+\f_B'(1,\l\,) C_1(\l)
=\l \cos \sqrt\l+\sqrt\l \sin \sqrt\l,\;\;\;\l\in\bC$}
\vskip 0.5mm
\noindent defines a meromorphic  Nevanlinna function $m$ such that the set $EV$ of all eigenvalues of the problem \eqref{5.67}, \eqref{5.68} coincides with the set  of all poles of $m$. Let $Z_\Phi$ and $Z_\Psi$ be the sets of all real zeros of $\Phi$ and $\Psi$ respectively. Assume also that
\begin{gather*}
S=\{s_k\}_0^\infty, \quad 0=s_0<s_1<\dots <s_k<\dots
\end{gather*}
is the set of all nonnegative solutions of the equation $s=-{\rm tg}\,s$. It is easy to see that $Z_\Psi=\{s_k^2\}_0^\infty$. Moreover, $Z_\Phi\cap Z_\Psi=\emptyset$ and hence $EV=Z_\Psi=\{s_k^2\}_0^\infty$. Next, the derivative $\Psi'$ is
\begin{gather*}
\Psi'(t)=\tfrac 1 2 (3\cos \sqrt t +\tfrac 1 {\sqrt t}\sin \sqrt t -\sqrt t \sin \sqrt t), \quad t\in(0,\infty).
\end{gather*}
Hence $\Psi'(0)=\lim\limits_{t\to +0}\Psi'(t)=2$ and the equality $\displaystyle  \frac {\sin s_k} {s_k}=-\cos s_k $ yields
\begin{gather*}
\Psi'(s_k^2)=\tfrac 1 2 (2\cos s_k-s_k\sin s_k), \quad k\in\bN.
\end{gather*}
Let $\wh\xi_k=\underset{s_k^2}{\rm res}\,m $. Then $\wh\xi_k=\displaystyle {{\Phi(s_k^2)}\over{\Psi' (s_k^2)}}$ and, consequently, $\wh\xi_0=- \tfrac 1 2$,
\begin{gather*}
\wh\xi_k=\frac{2(s_k\sin s_k-\cos s_k)}{2\cos s_k-s_k\sin s_k}=
\frac{2(s_k {\rm tg}\,s_k -1)}{2-s_k {\rm tg}\,s_k}=-2 \frac {s_k^2+1}{s_k^2+2}, \quad k\in\bN.
\end{gather*}
Hence  by \eqref{1.12.1} the eigenfunctions $y_k$ in the expansion \eqref{1.4} are
\begin{gather*}
y_0(x)=\frac 1 2 \wh y_0, \qquad y_k(x)=2\frac{s_k^2+1}{s_k^2+2}\wh y_k \cos (s_k x), \quad k\in\bN,
\end{gather*}
where $\wh y_k$ are the Fourier coefficients $y$:
\begin{gather*}
\wh y_0=\int_{[0,1]} y(x)\,dx, \qquad \wh y_k=\int_{[0,1]} y(x) \cos (s_k x)\,dx.
\end{gather*}
Note also that the  Nevanlinna function $\tau$ defined before Theorem \ref{th1.3} is $\tau(\l)=\l$. Hence $\cB_\infty=1$ and consequently Case 1 in Theorem \ref{th1.3} holds.

Now applying Theorems \ref{th1.2}, \ref{th1.3} and Remark \ref{rem1.4} we arrive at the following assertion.
\begin{assertion}\label{ass5.25}
Each function $y\in L^2([0,1])$ admits the representation
\begin{gather}\label{5.70}
y(x)=b_0+\sum_{k=1}^\infty b_k cos(s_k x),
\end{gather}
where $\{s_k\}_1^\infty \; (s_k<s_{k+1})$ is the  set of all positive solutions of the equation $s=-{\rm tg}\, s$ and
\begin{gather*}
b_0=\frac 1 2 \int_{[0,1]} y(x)\, dx, \qquad b_k=2\frac {s_k^2+1}{s_k^2+2} \int_{[0,1]} y(x) \cos (s_k x)\, dx
\end{gather*}
(the series in \eqref{5.70} converges in $L^2([0,1])$). If in addition $y,y'\in AC ([0,1]), \; y''\in L_2 ([0,1])$ and $y'(0)=0, \; y(b)=0$, then the series in  \eqref{5.70} converges uniformly on $[0,1]$.
\end{assertion}

\end{document}